\documentclass[12pt,a4paper]{article}

\usepackage{amssymb, amsmath, amsthm}
\usepackage{mathtools}

\usepackage[cm]{fullpage}
\usepackage[english]{babel}
\usepackage[cp1250]{inputenc}
\usepackage[T1]{fontenc}
\usepackage[pdftex]{graphicx}
\usepackage{booktabs}
\usepackage{bbm}
\usepackage{{enumerate}}

\usepackage[svgnames]{xcolor}
\usepackage[numbers]{natbib}

\newtheorem{thm}{Theorem}
\newtheorem{lem}{Lemma}
\newtheorem{prop}{Proposition}

\title{Numerical scheme for Erd\'elyi--Kober fractional diffusion equation using Galerkin--Hermite method}
\author{\L ukasz P\l ociniczak\thanks{Faculty of Pure and Applied Mathematics, Wroc{\l}aw University of Science and Technology, Wyb. Wyspia{\'n}skiego 27, 50-370 Wroc{\l}aw, Poland},
	Mateusz \'Swita{\l}a\footnotemark[1]$\;^,$\footnote{\underline{Corresponding Author}, e-mail: mateusz.switala@pwr.edu.pl}}
\date{}

\begin{document}
\maketitle

\begin{abstract}
The aim of this work is to devise and analyse an accurate numerical scheme to solve Erd\'elyi--Kober fractional diffusion equation. This solution can be thought as the marginal \emph{pdf} of the stochastic process called the \emph{generalized grey Brownian motion} (ggBm). The ggBm includes some well-known stochastic processes: Brownian motion, fractional Brownian motion and grey Brownian motion. To obtain convergent numerical scheme we transform the fractional diffusion equation into its weak form and apply the discretization of the Erd\'elyi--Kober fractional derivative. We prove the stability of the solution of the semi-discrete problem and its convergence to the exact solution. Due to the singular in time term appearing in the main equation the proposed method converges slower than first order. Finally, we provide the numerical analysis of the full-discrete problem using orthogonal expansion in terms of Hermite functions.
\\
	

\end{abstract}
\section{Introduction}
Fractional calculus is a branch of mathematics that is widely applied in other areas of science. Due to the property of nonlocality, fractional models remarkably well describe many natural phenomena, where for instance, some memory effects appear. Fractional operators are have been extensively analysed, both analytically and numerically. Some thorough expositions can be found for example in \cite{kiryakova1993,podlubny1998}. In the literature \cite{sun2018new} one can find a broad variety of applications of fractional models to real-world phenomena. Probably one of the most known example is the problem of anomalous diffusion where the use of fractional operators to describe this phenomenon accurately has met a large success \cite{metzler2000, mainardi2007, de2011, plociniczak2019derivation, plociniczak2015analytical}. In this work we focus on Erd\'elyi--Kober fractional operators that also are useful in certain physical situations \cite{sneddon1975, plociniczak2014approximation}. We consider two operators from this family: integral $I_\eta^{\gamma,\mu}$, and derivative $D_\eta^{\gamma,\mu}$ with parameters $\eta$, $\gamma$, $\mu$ to be defined in the next section. The Erd\'elyi--Kober fractional derivative also appears in the literature where the deterministic fractional diffusion equation describing marginal density of the certain stochastic processes is considered \cite{pagnini2012}. For the properties of the Erd\'elyi--Kober fractional operators the reader is invited to consult \cite{kilbas2006, kiryakova1993} and \cite{luchko2007} where the Caputo type modification of the differential operator has been considered. Moreover, in \cite{ibrahim2007,wang2012} one can find results concerning existence and uniqueness for integral equations with Erd\'elyi--Kober fractional operators.

The main motivation of our work is the Erd\'elyi--Kober fractional diffusion equation investigated in \cite{pagnini2012}
\begin{equation}\label{Eq:Main}
    \frac{\partial u}{\partial t}=\frac{\alpha}{\beta}t^{\alpha-1}D_{\alpha/\beta}^{\beta-1,1-\beta}\frac{\partial^2 u}{\partial x^2}, \quad (t,x)\in (0,T]\times \mathbb{R},
\end{equation}
with the initial condition $u(0,x)=u_0(x)$. Here, $0<\beta\le1,\ 0<\alpha<2$ and 
\begin{equation}
    D_{\alpha/\beta}^{\beta-1,1-\beta}:=[(\beta-1)+1+\beta/\alpha t]I_{\alpha/\beta}^{0,\beta},
\end{equation} 
is a Erd\'elyi--Kober differential operator. There exist a strong connection between solution of the above fractional integro-differential equation and stochastic processes. In particular, A. Mura \cite{mura2008} originally introduced the following integro-differential equation
\begin{equation}\label{Eq:Mura}
    u(t,x)=u_0(x)+\frac{\alpha}{\Gamma(\beta+1)}\int\limits_0^t \tau^{\alpha/\beta-1}(t^{\alpha/\beta}-\tau^{\alpha/\beta})^{\beta-1}\frac{\partial^2 u(\tau,x)}{\partial x^2}\, d \tau,
\end{equation}
where $u(t,x)$ is the one-point one-time density function of particle dispersion of a generalized grey Brownian motion (ggBm) with $0<\alpha<2,\ 0<\beta\le 1$. The detailed discussion on the meaning of \eqref{Eq:Mura} and the associated family of stochastic processes, denoted by ggBm, the reader can find in  \cite{mainardi2010Mura, mura2008, mura2009, mura2008character, mura2008non}. Choosing appropriate values for the parameters of ggBm, i.e. $\alpha,\ \beta$,  we can recover some well-known stochastic processes \cite{mura2008}: 
\begin{itemize}
    \item $\alpha=\beta=1$: standard Brownian motion,
    \item  $0<\alpha<2,\ \beta=1$: fractional Brownian motion,
    \item $\alpha=\beta$: grey Brownian motion.
\end{itemize}
Moreover, the ggBm includes the non-local stochastic models for anomalous diffusion: of both sub- ($0<\alpha<1$) and super- type ($1<\alpha<2$). Differentiating \eqref{Eq:Mura} with respect to time we obtain \eqref{Eq:Main}. Let us notice that putting $\beta=1$ in \eqref{Eq:Mura} and taking the first time derivative we arrive at the fractional Brownian diffusion equation \cite{watkins2009}
\begin{equation}
    \frac{\partial u(t,x)}{\partial t}=\alpha t^{\alpha-1}\frac{\partial^2 u(t,x)}{\partial x^2}.
\end{equation}

In this work we use the Galerkin method to devise a stable numerical scheme for solving \eqref{Eq:Main}. In order to do so, we multiply \eqref{Eq:Main} by a test function $v\in H^1(\mathbb{R})$ and integrate by parts to obtain the weak form 
\begin{equation}\label{Eq:WeakForm}
    (u_t,\chi)=-\frac{\alpha}{\beta}t^{\alpha-1}(D_{\alpha/\beta}^{\beta-1,1-\beta}u_x,\chi_x),\quad \forall \chi\in H^1(\mathbb{R}).
\end{equation}
To simplify notation we use $u_t\coloneqq \partial u(t,x)/\partial t$, and $u_x$ to denote the weak derivatives with respect appropriate variables. Furthermore, by $(\cdot,\cdot)$ and $\|\cdot\|$ we denote the standard $L^2(\mathbb{R})$ inner product and norm, respectively, i.e.
\begin{equation}
    (u,v)\coloneqq \int\limits_{-\infty}^{\infty} u(x)v(x)\,d x,\quad \|u\|=(u,u)^{1/2},\quad u,v\in L^2(\mathbb{R}).
\end{equation}
The Sobolev space $H^1(\mathbb{R})$ is defined in a standard way, i.e. $H^1\coloneqq \{u\in L^2(\mathbb{R}):D^1 u\in L^2(\mathbb{R}) \}$, where $D^1 u$ is a first order weak derivative of the function $u$. Hence, if $u$ is a solution of \eqref{Eq:WeakForm} then  $D^1 u=u_x$. Moreover, note that $H^1(\mathbb{R})=H_0^1(\mathbb{R})$.

Various numerical schemes have been proposed for solving the time fractional diffusion equation \cite{li2019numerical}. The procedure usually involves discretization of fractional operators with respect to time using for example, the L1 scheme \cite{oldham1974fractional}, convolution quadrature \cite{lubich2004convolution} or modifications of them. The spatial dimension can be tackled, apart from other approaches, by the Finite Difference \cite{zhang2011error,murio2008implicit}, Finite Element \cite{li2016analysis, jin2018analysis} or Spectral methods \cite{plociniczak2021linear, li2009space, lin2007finite}. The inherent characteristic feature of the solution to the time-fractional diffusion equation is its singularity near $t=0$ which is in a stark contrast with the classical case \cite{sakamoto2011initial}. For this reason the temporal discretization can experience accuracy loss \cite{stynes2021survey}. There are several methods to overcome this difficulty, one of which is the use of graded mesh \cite{gracia2018convergence, kopteva2019error}. Our initial studies indicate that that the solution Erd\'elyi--Kober diffusion equation \eqref{Eq:Main} also exhibits such a singular behaviour. 

The outline of the paper is organised as follows. In Sect. \ref{discretization} we propose the discretization methods for the Erd\'elyi--Kober fractional derivative together with the proofs of their errors asymptotic behaviour. In Sect. \ref{Results} we prove the stability of the semi-discrete problem of weak formulation of Erd\'elyi--Kober fractional diffusion equation \eqref{TimeDiscreteEq} and derive estimates on the error of exact solution approximation. In Sect. \ref{NumAnalysis}, the fully-discrete method is introduced. There, the orthogonal basis composed of Hermite functions is used to approximate the solution along the spatial dimension. Furthermore, numerical examples are given to support the theoretical results of proposed methods. Finally, in Sect. \ref{Conclusions}, the conclusions of our results are discussed.

\section{Discretization of the Erd\'elyi--Kober differential operator}
\label{discretization}
Following \cite{kiryakova1993, kiryakova1997} and \cite{kilbas2006} let us define the Erd\'elyi--Kober fractional integral operator $I_{\eta}^{\gamma, \mu}$
\begin{equation}
    I_{\eta}^{\gamma, \mu} \phi(t)\coloneqq\frac{\eta }{\Gamma(\mu)}t^{-\eta (\mu +\gamma)}\int\limits_{0}^t \tau^{\eta (\gamma+1)-1}(t^\eta-\tau^\eta)^{\mu-1}\phi(\tau)\, d \tau,\quad t\in [0,T],
\end{equation}
where $\mu >0,\ \eta>0,\ \gamma\in\mathbb{R}$, $\phi\in C_\lambda \coloneqq \{f(x)=x^p \hat{f}(x),\ p>\lambda, \ \hat{f}\in C[0,\infty)\}$  and $\lambda>-\beta(\gamma+1)$.
When $\eta=1$ the operator $I_\eta^{\gamma,\mu}$ becomes the fractional integral operator originally introduced in \cite{erdelyi1940} and \cite{kober1940}. Note that if we change the variable according to $x=\tau/t$ we get the equivalent form of the Erd\'elyi--Kober fractional integral which is particularly useful for numerical calculations
\begin{equation}\label{EK:Int}
    I_{\eta}^{\gamma, \mu} \phi(t)=\frac{\eta }{\Gamma(\mu)}\int\limits_{0}^1 x^{\eta (\gamma+1)-1}(1-x^\eta)^{\mu-1}\phi(x t)\, d x,\quad t\in [0,T].
\end{equation}
Moreover, using the above integral operator, we define the Erd\'elyi--Kober fractional differential operator \cite{kiryakova1993, kilbas2006},
\begin{equation}\label{EK:Diff}
    D_{\eta}^{\gamma,\mu}\phi(t)\coloneqq\prod\limits_{j=1}^n \biggl(\gamma+j+\frac{1}{\eta} t \frac{d}{d t}\biggr) (I_\eta^{\gamma+\mu, n-\mu} \phi(t)),\quad t\in [0,T],
\end{equation}
where $n-1<\mu<n,\ n\in \mathbb{N}$, $\phi\in C_\lambda^n \coloneqq \{f(x)=x^p \hat{f}(x),\ p>\lambda, \ \hat{f}\in C^n[0,\infty)\}$, and the parameters $\eta,\ \mu,\ \gamma,\ \lambda$ satisfy the same condition as before.

Next, we fix $t=t_n=n k$, where $k$ is a some small positive constant, and consider definition \eqref{EK:Int} with $\eta=\alpha/\beta,\ \gamma=0$ and $\mu=\beta$. To provide a discretization of $I_{\alpha/\beta}^{0,\beta}$ we divide the interval $[0,1]$ into $n$ equally spaced subintervals, and approximate the value of the integral on each subinterval by the rectangule rule obtaining
\begin{equation}
L_{\alpha/\beta}^{0,\beta}\phi(t_n)\coloneqq \sum\limits_{i=1}^{n} c_{n,i} \phi\biggl(\frac{i}{n} t_n\biggr)=\sum\limits_{i=1}^{n} c_{n,i} \phi(t_i),
\end{equation}
where
\begin{equation}\label{CoeffDef}
     c_{n,i}=\frac{1}{\beta \Gamma(\beta)} \biggl(\biggl(1-\biggl(\frac{i-1}{n}\biggr)^{\alpha /\beta }\biggr)^{\beta }-\biggl(1-\biggl(\frac{i}{n}\biggr)^{\alpha /\beta }\biggr)^{\beta}\biggr).
 \end{equation}
Furthermore, let us notice that putting $\eta=\alpha/\beta,\ \gamma=\beta-1$ and $\mu=1-\beta$ into \eqref{EK:Diff} we have
 \begin{equation}
     D_{\alpha/\beta}^{\beta-1,1-\beta}\phi(t_n)=\biggl(\beta+\frac{\beta}{\alpha}t_n\frac{d}{d t}\biggr)I_{\alpha/\beta}^{0,\beta}\phi(t_n).
 \end{equation}
Using the discretization operator $L_{\alpha/\beta}^{0,\beta}$ and the finite difference scheme for the first-order derivative, we obtain the discrete Erd\'elyi--Kober fractional differential operator 
 \begin{equation}
     G_{\alpha/\beta}^{\beta-1,1-\beta}\phi(t_n)\coloneqq\beta L_{\alpha/\beta}^{0,\beta}\phi(t_n) + \frac{\beta}{\alpha} t_n \overline{\partial} L_{\alpha/\beta}^{0,\beta}\phi(t)=\beta L_{\alpha/\beta}^{0,\beta}\phi(t_n) + \frac{\beta}{\alpha} t_n \frac{1}{k}( L_{\alpha/\beta}^{0,\beta}\phi(t_n)-L_{\alpha/\beta}^{0,\beta}\phi(t_{n-1})).
 \end{equation}
However, the above notion of the discretization of Erd\'elyi--Kober fractional differential operator is not the only one. Let us notice that the derivative part of the operator $D_{\alpha/\beta}^{\beta-1,1-\beta}$ can be rewritten in the following way
 \begin{equation}
 \begin{split}
     t_n \frac{d}{d t} I_{\alpha/\beta}^{0,\beta}\phi(t_n) &= \frac{\alpha}{\beta \Gamma(\beta)} t_n \frac{d}{d t} \int\limits_{0}^{1}\tau^{\frac{\alpha}{\beta}-1}\bigl(1-\tau^{\frac{\alpha}{\beta}}\bigr)^{\beta-1}\phi(\tau t_n)\, d \tau\\
     &=\frac{\beta}{\alpha} \frac{\alpha}{\beta \Gamma(\beta)} \int\limits_{0}^{1}\tau^{\frac{\alpha}{\beta}-1}\bigl(1-\tau^{\frac{\alpha}{\beta}}\bigr)^{\beta-1}t_n \frac{d}{d t}\phi(\tau t_n)\, d \tau \\
     &=\frac{\beta}{\alpha}\frac{\alpha}{\beta \Gamma(\beta)} \sum\limits_{i=1}^{n}\int\limits_{\frac{i-1}{n}}^{\frac{i}{n}}\tau^{\frac{\alpha}{\beta}-1}\bigl(1-\tau^{\frac{\alpha}{\beta}}\bigr)^{\beta-1} \tau  \frac{d}{d \tau}\phi(\tau t_n)\, d \tau.
     \end{split}
 \end{equation}
Hence, on each subinterval we approximate the derivative with respect to $\tau$  by a finite difference scheme and then the alternative discretization of the operator $D_{\alpha/\beta}^{\beta-1,1-\beta}$ is
 \begin{equation}
 \begin{split}
     K_{\alpha/\beta}^{\beta-1,1-\beta}\phi(t_n)\coloneqq & \beta L_{\alpha/\beta}^{0,\beta}\phi(t_n)+\frac{\beta}{\alpha}\frac{\alpha}{\beta \Gamma(\beta)} \sum\limits_{i=1}^{n} d_{n,i}  n (\phi(t_i)-\phi(t_{i-1})\\
     =& \beta L_{\alpha/\beta}^{0,\beta}\phi(t_n)+\frac{\beta}{\alpha}\frac{\alpha}{\beta \Gamma(\beta)} \sum\limits_{i=2}^{n} (d_{n,i-1}-d_{n,i})  n \phi(t_{i-1}) + n d_{n,n} \phi(t_n) - n d_{n,1} \phi(0),
     \end{split}
 \end{equation}
 with
 \begin{equation}
        \begin{split} d_{n,i}=&\frac{\beta}{n (\alpha +\beta )}  \biggl(n \biggl(\frac{i}{n}\biggr)^{\frac{\alpha}{\beta}+1 } \, _2F_1\biggl(1-\beta ,1+\frac{\beta }{\alpha };2+\frac{\beta }{\alpha };\biggl(\frac{i}{n}\biggr)^{\alpha /\beta }\biggr)\\
         \ & -n \biggl(\frac{i-1}{n}\biggr)^{\frac{\alpha }{\beta }+1} \, _2F_1\biggl(1-\beta ,1+\frac{\beta }{\alpha };2+\frac{\beta }{\alpha };\biggl(\frac{i-1}{n}\biggr)^{\alpha /\beta }\biggr)\biggr),
         \end{split}
     \end{equation}
where $ _2 F_1(a,b;c;z)$ is the Gauss hypergeometric function
 \begin{equation}
        \, _2 F_1 (a,b;c;z)=\sum\limits_{k=0}^{\infty}\frac{(a)_k (b)_k}{(c)_k}\frac{z^k}{k!},
    \end{equation}
    with Pochhammer symbol
    \begin{equation}
        (a)_k=\left\{\begin{array}{ll}
            1 & \textrm{for } k=0, \\
            a(a+1)\cdots (a+n-1) &\textrm{for } k>0. 
        \end{array}\right.
    \end{equation}
The evaluation of the Gauss hypergeometric function may be expensive in practice, hence in further numerical analysis we would rather use $G_{\alpha/\beta}^{\beta-1,1-\beta}$ than $K_{\alpha/\beta}^{\beta-1,1-\beta}$ .

The following theorem provides the estimates on the order of discretization errors of the approximation operators $L_{\alpha/\beta}^{0,\beta},\ K_{\alpha/\beta}^{\beta-1,1-\beta}$ $G_{\alpha/\beta}^{\beta-1,1-\beta}$ as the number of the subintervals of $[0,1]$ goes to infinity.
\begin{thm}\label{Thm:Order}
 Fix $0<\alpha,\beta<1$ and assume that $\phi \in C^1([0,T])$ and $\psi \in C^2([0,T])$. Then, for a fixed $t_n \in (0,T)$, where $t_n=n k$, the discretization errors corresponding to the operator $ L_{\alpha/\beta}^{0,\beta},\ K_{\alpha/\beta}^{\beta-1,1-\beta}$ $G_{\alpha/\beta}^{\beta-1,1-\beta}$ can be estimated as below.
 \begin{itemize}
     \item Integral operator
     \begin{equation}\label{EK:ErrorEstimate}
        I_{\alpha/\beta}^{0,\beta}\phi(t_n) - L_{\alpha/\beta}^{0,\beta}\phi(t_n) =\frac{t_n}{n} \frac{ \phi'(\sigma t_n)}{\beta \Gamma(\beta)} .
     \end{equation}
     \item Differential operator I
     \begin{equation}\label{EKDiff:ErrorEstimate}
     \begin{split}
         |D_{\alpha/\beta}^{\beta-1,1-\beta} \psi(t_n)-G_{\alpha/\beta}^{\beta-1,1-\beta}\psi(t_n)|\le & \frac{1}{n}\biggl[ \beta t_n \frac{ |\psi'(\sigma t_n)|}{\beta \Gamma(\beta)}+\frac{t_n \beta \Gamma \biggl(\frac{2 \beta }{\alpha }+1\biggr)}{\alpha \Gamma \biggl(\frac{2 \beta }{\alpha }+\beta +1\biggr)}  |\psi''(t^*)| \\
         \ & + t_n C_1 |\psi'(t^{* *})|+\frac{t_n^2 C_2}{\beta \Gamma(\beta)} |\psi''(t^{* * *})|\biggr].
         \end{split}
     \end{equation}
     \item Differential Operator II
     \begin{equation}
        | D_{\alpha/\beta}^{\beta-1,1-\beta} \psi(t_n)-K_{\alpha/\beta}^{\beta-1,1-\beta}\psi(t_n)|\le \frac{1}{n}\biggl(\frac{t_n |\psi'(\sigma t_n)|}{\beta \Gamma(\beta)}  + \frac{\Gamma \biggl(1+\frac{\beta }{\alpha }\biggr)}{\Gamma \biggl(\frac{\beta }{\alpha }+\beta +1\biggr)} t_n^2 |\psi''(\sigma^* t_n)| \biggr).
     \end{equation}
 \end{itemize}
 Here, $\sigma,\, \sigma^*\in (0,1),\ t^{* },\, t^{* * },\, t^{* * *} \in (0,t_n)$ and $C_1,\ C_2$ are some positive constants independent on $n$. 
 \end{thm}
 \begin{proof}
 The asymptotic relation for Erd\'elyi--Kober fractional integral operator $I_{\alpha/\beta}^{0,\beta}$ was delivered in \cite{plociniczak2017}, therefore, we will not provide a detailed proof of it here.

Since we know the order of discretization of Erd\'elyi--Kober fractional integral operator, we proceed to find the order of discrete representations of the operator $D_{\alpha/\beta}^{\beta-1,1-\beta}$. In the beginning, let us consider the operator $G_{\alpha/\beta}^{\beta-1,1-\beta}$. Note that Erd\'elyi--Kober fractional differential operator can be rewritten as
 \begin{equation}
 \begin{split}
     D_{\alpha/\beta}^{\beta-1,1-\beta}\psi(t_n)= & \beta I_{\alpha/\beta}^{0,\beta}\psi(t_n)+\frac{\beta}{\alpha}t_n\frac{d}{d t}I_{\alpha/\beta}^{0,\beta}\psi(t_n) =\beta L_{\alpha/\beta}^{0,\beta}\psi(t_n) +\beta R_n+\frac{\beta}{\alpha}t_n\frac{d}{d t}I_{\alpha/\beta}^{0,\beta}\psi(t_n)\\
         = & \beta L_{\alpha/\beta}^{0,\beta}\psi(t_n) +\beta R_n + \frac{\beta}{\alpha} t_n \overline{\partial} L_{\alpha/\beta}^{0,\beta}\psi(t_n) + \frac{\beta}{\alpha}\widehat{R_n},
        \end{split}
 \end{equation}
 where $R_n$ is a discretization error of the Erd\'elyi--Kober fractional integral operator  \eqref{EK:ErrorEstimate} and $\widehat{R_n}= t_n (\frac{d}{d t}I_{\alpha/\beta}^{0,\beta}\psi(t_n) - \overline{\partial}I_{\alpha/\beta}^{0,\beta}\psi(t_n)+ \overline{\partial}I_{\alpha/\beta}^{0,\beta}\psi(t_n)- \overline{\partial}L_{\alpha/\beta}^{0,\beta}\psi(t_n))$. First difference in $\widehat{R_n}$ can be estimated in a standard way
 \begin{equation}
    \begin{split}
        \frac{d}{d t}I_{\alpha/\beta}^{0,\beta}\psi(t_n) - \overline{\partial}I_{\alpha/\beta}^{0,\beta}\psi(t_n)= &  \frac{d}{d t}I_{\alpha/\beta}^{0,\beta}\psi(t_n) - \frac{1}{k}(I_{\alpha/\beta}^{0,\beta}\psi(t_n)-I_{\alpha/\beta}^{0,\beta}\psi(t_{n-1})) = \frac{1}{n} \frac{d^2}{d t^2} I_{\alpha/\beta}^{0,\beta}\psi(t^{*})\\
        \  = & \frac{\Gamma \biggl(\frac{2 \beta }{\alpha }+1\biggr)}{\Gamma \biggl(\frac{2 \beta }{\alpha }+\beta +1\biggr)} \psi''(\sigma^* t^*) \frac{1}{n},
        \end{split}
    \end{equation}
    where $\sigma^* \in (0,1)$ and $t^* \in [t_{n-1},t_n]$. However, to estimate the order of the second term in $\widehat{R_n}$ we need to investigate this term carefully. Note that using again the Mean Value Theorem for Integrals and Sums, and the relation $c_{n,i}<c_{n-1,i}$ for $i\in \{1,2,\ldots,n-1\}$ we have
    \begin{equation}
    \begin{split}
         t_n\overline{\partial}I_{\alpha/\beta}^{0,\beta}\psi(t_n)  - t_n\overline{\partial}L_{\alpha/\beta}^{0,\beta}\psi(t_n) 
         = &  \frac{t_n}{k}\sum\limits_{i=1}^{n-1} \frac{\alpha}{\beta \Gamma(\beta)}\int\limits_{i-1}^{i}(f(n,s)-f(n-1,s))(\psi(i k)-\psi(s k))\, d s\\
    \ & + \frac{t_n}{k}\frac{\alpha}{\beta \Gamma(\beta)}\int\limits_{n-1}^{n} f(n,s)(\psi(n k)-\psi(s k))\, d s\\
    = & t_n \psi'(x k) \sum\limits_{i=1}^{n-1} \frac{\alpha}{\beta \Gamma(\beta)}\int\limits_{i-1}^{i}(f(n,s)-f(n-1,s))(i-s)\, d s\\
    \ & + t_n \psi'(x_n k)\frac{\alpha}{\beta \Gamma(\beta)}\int\limits_{n-1}^{n}f(n,s)(n-s)\, d s,
\end{split}
\end{equation}
where $i-1\le s<\sigma_i <i$, $0<x<n-1,\ n-1<x_n<n$ and
\begin{equation}
    f(n, s)\coloneqq \frac{1}{n}\biggl(\frac{s}{n}\biggr)^{\frac{\alpha}{\beta}-1} \biggl(1-\biggl(\frac{s}{n}\biggr)^{\frac{\alpha}{\beta}}\biggr)^{\beta-1}.
\end{equation}
 Furthermore, performing a straightforward calculation we get
 \begin{equation}
     C_n\coloneqq\sum\limits_{i=1}^{n} \frac{\alpha}{\beta \Gamma(\beta)}\int\limits_{i-1}^{i}f(n,s) (i-s) \, d s\le \sum\limits_{i=1}^{n} \frac{\alpha}{\beta \Gamma(\beta)}\int\limits_{i-1}^{i}f(n,s) \, d s=\frac{1}{\beta \Gamma(\beta)}.
 \end{equation}
 Moreover, note that $C_n$ can also be written in the following way
 \begin{equation}
 \label{Form1}
     C_n= \sum\limits_{i=0}^{n-1} \frac{i}{\beta \Gamma(\beta)} \biggl(1-\biggl(\frac{i}{n}\biggr)^{\alpha/\beta}\biggr)^\beta - \frac{n \Gamma \biggl(\frac{\alpha +\beta }{\alpha }\biggr)}{\Gamma \biggl(\frac{\beta }{\alpha }+\beta +1\biggr)}.
 \end{equation}
 Next, we use the Intermediate Value Theorem to get
 \begin{equation}\label{IntValueThoremPsi}
\begin{split}
    \psi'(x k) \sum\limits_{i=1}^{n-1} \frac{\alpha}{\beta \Gamma(\beta)}\int\limits_{i-1}^{i}f(n,s) (i-s)\, d s + \psi'(x_n k) \frac{\alpha}{\beta \Gamma(\beta)}\int\limits_{n-1}^{n}f(n,s)(n-s)\, d s =\psi'(x^* k) C_n,
    \end{split}
\end{equation}
where $x^*\in (x, x_n)$. 
Finally, we have
 \begin{equation}
\begin{split}
    t_n\overline{\partial}I_{\alpha/\beta}^{0,\beta}\psi(t_n) - t_n\overline{\partial}L_{\alpha/\beta}^{0,\beta}\psi(t_n) & = t_n (\psi'(x^* k) C_n - \psi'(x k) C_{n-1}) \\
    \  &= t_n \psi'(x^* k) (C_n-C_{n-1})+ t_n C_n (\psi'(x^* k)-\psi'(x k))\\
    \ & = t_n \psi'(x^* k) (C_n-C_{n-1})+ t_n C_n (x^*-x) k \psi''(x^{* *} k),
    \end{split}
\end{equation}
where $x^{* *}\in (x,x^*)$. Let us notice that the difference $|x-x^*|$ does not increase when $n$ becomes larger: $\psi$ is continuously differentiable on the interval $(0,t_n]$ and due to the fact that for increasing $n$ the contribution of the integral in second term in \eqref{IntValueThoremPsi} to $C_n$ becomes negligible. Next, we use \eqref{Form1} to obtain
\begin{equation}
    \begin{split}
        \frac{\beta \Gamma(\beta)}{\alpha}(C_n-C_{n-1})= &  \sum\limits_{i=0}^{n-1}  \frac{1}{\alpha}\biggl(1-\biggl(\frac{i}{n}\biggr)^{\alpha/\beta}\biggr)^\beta-\frac{1}{\alpha}\biggl(1-\biggl(\frac{i}{n-1}\biggr)^{\alpha/\beta}\biggr)^\beta-\frac{\beta \Gamma(\beta) \Gamma \biggl(1+\frac{\beta }{\alpha }\biggr)}{\alpha \Gamma \biggl(\frac{\beta }{\alpha }+\beta +1\biggr)}\\
        = &  \sum\limits_{i=0}^{n-1} \biggl[ \frac{1}{\alpha}\biggl(1-\biggl(\frac{i}{n}\biggr)^{\alpha/\beta}\biggr)^\beta-\frac{1}{\alpha}\biggl(1-\biggl(\frac{i}{n-1}\biggr)^{\alpha/\beta}\biggr)^\beta-\int\limits_{\frac{i}{n}}^{\frac{i+1}{n}} s^{\frac{\alpha}{\beta}}(1-s^{\frac{\alpha}{\beta}})^{\beta-1} \, d s\biggr]\\
        = & \sum\limits_{i=0}^{n-1} \biggl[ \frac{1}{\alpha}\biggl(1-\biggl(\frac{i}{n}\biggr)^{\alpha/\beta}\biggr)^\beta-\frac{1}{\alpha}\biggl(1-\biggl(\frac{i}{n-1}\biggr)^{\alpha/\beta}\biggr)^\beta-\frac{i}{n}\int\limits_{\frac{i}{n}}^{\frac{i+1}{n}} s^{\frac{\alpha}{\beta}-1}(1-s^{\frac{\alpha}{\beta}})^{\beta-1} \, d s\biggr]\\
        \ & +  \sum\limits_{i=0}^{n-1} \biggl[\frac{i}{n}\int\limits_{\frac{i}{n}}^{\frac{i+1}{n}} s^{\frac{\alpha}{\beta}-1}(1-s^{\frac{\alpha}{\beta}})^{\beta-1} \, d s-\int\limits_{\frac{i}{n}}^{\frac{i+1}{n}} s^{\frac{\alpha}{\beta}}(1-s^{\frac{\alpha}{\beta}})^{\beta-1} \, d s\biggr]=S_1 + S_2.
    \end{split}
\end{equation}
Note, that the absolute value of the second term can be easily bounded from above as follows
\begin{equation}
\begin{split}
    |S_2|=&\sum\limits_{i=0}^{n-1} \biggl|\frac{i}{n}\int\limits_{\frac{i}{n}}^{\frac{i+1}{n}} s^{\frac{\alpha}{\beta}-1}(1-s^{\frac{\alpha}{\beta}})^{\beta-1} \, d s-\int\limits_{\frac{i}{n}}^{\frac{i+1}{n}} s^{\frac{\alpha}{\beta}}(1-s^{\frac{\alpha}{\beta}})^{\beta-1} \, d s\biggr|\\
    = & \sum\limits_{i=0}^{n-1}\int\limits_{\frac{i}{n}}^{\frac{i+1}{n}} s^{\frac{\alpha}{\beta}-1}(1-s^{\frac{\alpha}{\beta}})^{\beta-1} (s-\frac{i}{n})\, d s\le \frac{1}{n} \sum\limits_{i=0}^{n-1} \int\limits_{\frac{i}{n}}^{\frac{i+1}{n}} s^{\frac{\alpha}{\beta}-1}(1-s^{\frac{\alpha}{\beta}})^{\beta-1} \, d s=\frac{1}{\alpha} \frac{1}{n}.
    \end{split}
\end{equation}
Considering now $S_1$, we evaluate the integral for each $i$ to get
\begin{equation}
\begin{split}
    S_1= &
    \sum\limits_{i=1}^n \biggl[ \frac{1}{\alpha}\biggl(1-\biggl(\frac{i}{n}\biggr)^{\alpha/\beta}\biggr)^\beta-\frac{1}{\alpha}\biggl(1-\biggl(\frac{i}{n-1}\biggr)^{\alpha/\beta}\biggr)^\beta \\
    \ & -\frac{i}{n} \biggl(\frac{1}{\alpha}\biggl(1-\biggl(\frac{i}{n}\biggr)^{\alpha /\beta }\biggr)^{\beta }-\frac{1}{\alpha}\biggl(1-\biggl(\frac{i+1}{n}\biggr)^{\alpha /\beta }\biggr)^{\beta }\biggr)\biggr]
    =\sum\limits_{i=1}^n \theta_i.
    \end{split}
\end{equation}
Next, for each $i\in\{1,2,\ldots,n-2\}$ we expand $\theta_i$ in the Taylor series for large $n$ with fixed $i/n$
\begin{equation}
\begin{split}
    \theta_i= &-\frac{1}{2 \beta  n^2}\biggl(1-\biggl(\frac{i}{n}\biggr)^{\alpha /\beta }\biggr)^{\beta -2} \biggl(\frac{i}{n}\biggr)^{\frac{\alpha }{\beta }-1} \biggl(\alpha -\beta +\beta  \biggl(\alpha  \biggl(\frac{i}{n}-1\biggr)+\frac{i}{n}+1\biggr) \biggl(\frac{i}{n}\biggr)^{\alpha /\beta }-\frac{\alpha  i}{n}-\frac{\beta  i}{n}\biggr)\\
    \ & +\mathcal{O}\biggl(\frac{(1-\bigl(\frac{i}{n}\bigr))^{\beta-2} \bigl(\frac{i}{n}\bigr)^{\frac{\alpha}{\beta}-2}}{n^3}\biggr).
\end{split}
\end{equation}
Let us notice that for sufficiently large $n$ we can bound $|\theta_i|$ from above as follows
\begin{equation}
\begin{split}
    |\theta_i|\le & C \frac{1}{2 \beta  n^2}\biggl(1-\biggl(\frac{i}{n}\biggr)^{\alpha /\beta }\biggr)^{\beta -2} \biggl(\frac{i}{n}\biggr)^{\frac{\alpha }{\beta }-1} \biggl|\biggl(\alpha -\beta +\beta  \biggl(\alpha  \biggl(\frac{i}{n}-1\biggr)+\frac{i}{n}+1\biggr) \biggl(\frac{i}{n}\biggr)^{\alpha /\beta }-\frac{\alpha  i}{n}-\frac{\beta  i}{n}\biggr)\biggr|\\
    = & C \frac{1}{ n^2} W\biggl(\frac{i}{n}\biggr),
    \end{split}
\end{equation}
where $C$ does not depend on $n$. Moreover, it is easy to note that
\begin{equation}
   \frac{1}{n} \sum\limits_{i=1}^{n-2} W\left(\frac{i}{n}\right)\xrightarrow{n\to \infty} \int\limits_{0}^{1}W(x)\, d x,
\end{equation}
where the convergence is a result of the definition of Riemann sum with the set $\{i/n:1\le i\le n-2\}$ as a partition of $(0,1)$ and the integrability of the function under integral sign. Function $W(x)$ has two singularities: at $x=0$ and $x=1$. Performing standard calculation we get $W(x)=(\alpha-\beta)x^{\alpha/\beta-1}+\mathcal{O}(x^{2\alpha/\beta-1})$ as $x\to 0$ and $W(x)=\alpha(1+\beta)(1-x)^{\beta-1}+\mathcal{O}((1-x)^\beta)$ as $x\to 1$. Since $\alpha,\beta>0$, the function $W(x)$ is integrable around its singular points. Furthermore, it is easy to see that 
\begin{equation}
    \theta_0=\theta_{n-1}=\mathcal{O}\biggl(\frac{1}{n}\biggr), \quad \mathrm{as}\ n\to \infty.
\end{equation}
Based on the above considerations we get 
\begin{equation}
    C_n-C_{n-1}=\mathcal{O}\biggl(\frac{1}{n}\biggr),\quad \mathrm{as}\ n\to \infty.
\end{equation}
Finally, we have
\begin{equation}
    |t_n\overline{\partial}I_{\alpha/\beta}^{0,\beta}\psi(t_n) - t_n\overline{\partial}L_{\alpha/\beta}^{0,\beta}\psi(t_n)|\le t_n \psi'(x^* k) \frac{C_1}{n}+t_n^2 C_n C_2 \psi''(x^{* *}k) \frac{1}{n},
\end{equation}
where $C_1$ and $C_2$ do not depend on $n$. Taking into account the estimates of all components in $\widehat{R_n}$ we further obtain
\begin{equation}
\begin{split}
    |\widehat{R_n}|= & |t_n (\frac{d}{d t}I_{\alpha/\beta}^{0,\beta}\psi(t_n) - \overline{\partial}I_{\alpha/\beta}^{0,\beta}\psi(t_n)+ \overline{\partial}I_{\alpha/\beta}^{0,\beta}\psi(t_n)- \overline{\partial}L_{\alpha/\beta}^{0,\beta}\psi(t_n))|\\
    \ \le & \frac{\Gamma \biggl(\frac{2 \beta }{\alpha }+1\biggr)}{\Gamma \biggl(\frac{2 \beta }{\alpha }+\beta +1\biggr)} t_n |\psi''(\sigma^* t^*)| \frac{1}{n} + t_n |\psi'(x^* k)| \frac{C_1}{n}+t_n^2 C_n C_2 |\psi''(x^{* *}k)| \frac{1}{n}\\
    \le & \frac{1}{n}\biggl[ \frac{t_n \Gamma \biggl(\frac{2 \beta }{\alpha }+1\biggr)}{ \Gamma \biggl(\frac{2 \beta }{\alpha }+\beta +1\biggr)} |\psi''(t^{*})| + t_n C_1 \psi'(t^{* *})|+\frac{t_n^2 C_2}{\beta \Gamma(\beta)} |\psi''(t^{* * *})|\biggr],
    \end{split}
\end{equation}
where $t^{*},t^{* *},t^{* * *} \in (0,t_n)$.
The above estimate together with $R_n$ yields \eqref{EKDiff:ErrorEstimate}. 

Let us now consider the operator $K_{\alpha/\beta}^{\beta-1,1-\beta}$. Similarly to the above considerations, the integral part of $D_{\alpha/\beta}^{\beta-1,1-\beta}$ is approximated using $\beta L_{\alpha/\beta}^{0,\beta} $. Now, let us examine solely the differentiation part of the analysed operator. Performing appropriate transformation we have
 \begin{equation}
 \begin{split}
     \frac{\alpha}{\beta \Gamma(\beta)} \sum\limits_{i=1}^{n} &  d_{n,i}  n (\psi(t_i)-  \psi(t_{i-1}))  - t_n \frac{d}{d t} I_{\alpha/\beta}^{0,\beta}\psi(t_n)  \\
     = & \frac{\alpha}{\beta \Gamma(\beta)}\biggl( \sum\limits_{i=1}^{n}\int\limits_{\frac{i-1}{n}}^{\frac{i}{n}}\tau^{\frac{\alpha}{\beta}-1}\bigl(1-\tau^{\frac{\alpha}{\beta}}\bigr)^{\beta-1} \tau  n (\psi(t_i)-\psi(t_{i-1}))\, d \tau
     -t_n \frac{d}{d t} \int\limits_{0}^{1}\tau^{\frac{\alpha}{\beta}-1}\bigl(1-\tau^{\frac{\alpha}{\beta}}\bigr)^{\beta-1}\psi(\tau t_n)\, d \tau\biggr)\\
     = & \frac{\alpha}{\beta \Gamma(\beta)}\biggl( \sum\limits_{i=1}^{n}\int\limits_{\frac{i-1}{n}}^{\frac{i}{n}}\tau^{\frac{\alpha}{\beta}-1}\bigl(1-\tau^{\frac{\alpha}{\beta}}\bigr)^{\beta-1}\tau\biggl[  n (\psi(t_i)-\psi(t_{i-1}))- \frac{d}{d \tau}\psi(\tau t_n)\biggr]\, d \tau\biggr).
     \end{split}
 \end{equation}
 We use Taylor series expansion for $\psi(x t_n)$ at point $x=\tau$  and obtain
 \begin{equation}
     \psi(x t_n)=\psi(\tau t_n)+ \frac{d }{d x}\psi(\tau t_n) (x-\tau) + \frac{t_n^2}{2} \psi''(\xi t_n) (x-\tau)^2,\quad \xi \in (\tau,x).
 \end{equation}
 Therefore, the expression in the square bracket under the integral sign can be rewritten in the following way
 \begin{equation}
 \begin{split}
    n (\psi(t_i)-\psi(t_{i-1}))- \frac{d}{d \tau}\psi(\tau t_n)= & n  \frac{t_n^2}{2}\biggl( \psi''(\xi_i t_n) \biggl(\frac{i}{n}-\tau\biggr)^2-\psi''(\xi_{i-1} t_n) \biggl(\frac{i-1}{n}-\tau\biggr)^2\biggr),
    \end{split}
 \end{equation}
 and
 \begin{equation}
     \begin{split}
         \bigl| \frac{\alpha}{\beta \Gamma(\beta)} \sum\limits_{i=1}^{n} & d_{n,i} n (\psi(t_i)-\psi(t_{i-1})-t_n \frac{d}{d t} I_{\alpha/\beta}^{0,\beta}\psi(t_n)  \bigr|\\
         \ & \le \frac{\alpha}{\beta \Gamma(\beta)}\biggl( \sum\limits_{i=1}^{n}\int\limits_{\frac{i-1}{n}}^{\frac{i}{n}}\tau^{\frac{\alpha}{\beta}}\bigl(1-\tau^{\frac{\alpha}{\beta}}\bigr)^{\beta-1}\biggl|  n  \frac{t_n^2}{2}\biggl( \psi''(\xi_i t_n) \biggl(\frac{i}{n}-\tau\biggr)^2-\psi''(\xi_{i-1} t_n) \biggl(\frac{i-1}{n}-\tau\biggr)^2\biggr)\biggr|\, d \tau\biggr)\\
         \ & \le \frac{\Gamma \biggl(1+\frac{\beta }{\alpha }\biggr)}{\Gamma \biggl(\frac{\beta }{\alpha }+\beta +1\biggr)} \frac{t_n^2 |\psi''(\sigma^* t_n)|}{n},\quad \sigma^*\in (0,1).
     \end{split}
 \end{equation}
 The above inequalities end the proof.
 \end{proof}
It is worth to mention that in \cite{plociniczak2017} authors proposed different discretization methods of the Erd\'elyi--Kober fractional integral operator. In addition to the rectangle rule they used also mid-point and trapezoid rule to obtain more accurate approximations.
In the convergence analysis of the approximate solution of \eqref{Eq:WeakForm} the functions $\phi$ and $\psi$ do not have to be continuously differentiable at $t=0$. Therefore, to tackle this singular behaviour we propose the weaker form of Theorem \ref{Thm:Order}.
 \begin{prop}\label{PropSingular}
 Fix $0<\alpha,\beta<1$  and assume that $\phi \in C^1((0,T))$ and $\psi \in C^2((0,T))$ such that
 \begin{align}
     |\phi(t)|+t^{1-\alpha} |\phi'(t)| & \le C,\\ \label{EK:CondPsi}
     |\psi(t)|+t^{1-\alpha} |\psi'(t)|+t^{2-\alpha} |\psi''(t)| & \le C.
 \end{align}
 Then, for a fixed $t_n \in (0,T)$, where $t_n=n k$, the discretization errors corresponding to the operators $ L_{\alpha/\beta}^{0,\beta},\ K_{\alpha/\beta}^{\beta-1,1-\beta}$ $G_{\alpha/\beta}^{\beta-1,1-\beta}$ can be estimated from above as below
  \begin{itemize}
     \item Integral operator
     \begin{equation}\label{EK:ErrorEstimate}
        |I_{\alpha/\beta}^{0,\beta}\phi(t_n) - L_{\alpha/\beta}^{0,\beta}\phi(t_n)| \le C\frac{t_n^\alpha}{n^{\min\{\alpha/\beta+\alpha,1\}}}.
     \end{equation}
     \item Differential operator I
     \begin{equation}\label{EKDiff:ErrorEstimate}
         |D_{\alpha/\beta}^{\beta-1,1-\beta} \psi(t_n)-G_{\alpha/\beta}^{\beta-1,1-\beta}\psi(t_n)|\le C t_n^\alpha \biggl(\frac{1}{n^{\min\{\alpha/\beta+\alpha,1\}}}+\frac{1}{n^{\max \{\beta,\alpha/\beta+\alpha-1\}}}\biggr).
     \end{equation}
     \item Differential Operator II
     \begin{equation}
        | D_{\alpha/\beta}^{\beta-1,1-\beta} \psi(t_n)-K_{\alpha/\beta}^{\beta-1,1-\beta}\psi(t_n)|\le C t_n^\alpha \biggl(\frac{1}{n^{\min\{\alpha/\beta+\alpha,1\}}}+\frac{1}{n}\biggr).
     \end{equation}
 \end{itemize}
 \end{prop}
 \begin{proof}
Let us consider first the integral operator $I_{\alpha/\beta}^{0,\beta}$. We estimate the error similarly as in the proof of Theorem \ref{Thm:Order}. The exception is the neighbourhood of point $t=0$, where we used the fact that term $t^{1-\alpha} |\phi'(t)|$ is bounded. Hence, we have
\begin{equation}
 \begin{split}
     |I_{\alpha/\beta}^{0,\beta}\phi(t_n) - L_{\alpha/\beta}^{0,\beta}\phi(t_n)|= & \biggl| \sum\limits_{i=1}^{n} \int\limits_{\frac{i-1}{n}}^{\frac{i}{n}} \tau^{\frac{\alpha}{\beta}-1}(1-\tau^{\frac{\alpha}{\beta}})^{\beta-1}\biggl(\phi(\tau t_n)-\phi\biggl(\frac{i}{n} t_n\biggr)\biggr)\, d \tau\biggr|\\
 \le & C \int\limits_0^\frac{1}{n} \tau^{\frac{\alpha}{\beta}-1}(1-\tau^{\frac{\alpha}{\beta}})^{\beta-1}\int\limits_{\tau t_n}^k|\phi'(x)|\, d x\, d  \tau  \\
 \ & +Ct_n^{\alpha}  \frac{1}{n} \int\limits_{\frac{1}{n}}^{1} \tau^{\frac{\alpha}{\beta}-1}(1-\tau^{\frac{\alpha}{\beta}})^{\beta-1} \tau^{\alpha-1}\, d \tau \\
 \le & C t_n^\alpha\int\limits_0^\frac{1}{n} \tau^{\frac{\alpha}{\beta}-1}(1-\tau^{\frac{\alpha}{\beta}})^{\beta-1}\biggl(\frac{1}{n^\alpha}-\tau^\alpha\biggr)\, d \tau  \\
 \ & +C t_n^{\alpha}  \frac{1}{n} \int\limits_{\frac{1}{n}}^{1} \tau^{\frac{\alpha}{\beta}-1}(1-\tau^{\frac{\alpha}{\beta}})^{\beta-1} \tau^{\alpha-1}\, d \tau\le  C\frac{t_n^\alpha}{n^{\min\{\alpha/\beta+\alpha,1\}}}.
     \end{split}
\end{equation}
Considering now $G_{\alpha/\beta}^{\beta-1,1-\beta}$ and, in particular, components of the remainder $\widehat{R_n}$ we have
 \begin{equation}
 \begin{split}
     \biggl| \frac{d}{d t}I_{\alpha/\beta}^{0,\beta}\psi(t_n) - \overline{\partial}I_{\alpha/\beta}^{0,\beta}\psi(t_n)\biggr|= &  \biggl|\frac{d}{d t}I_{\alpha/\beta}^{0,\beta}\psi(t_n) - \frac{1}{k}(I_{\alpha/\beta}^{0,\beta}\psi(t_n)-I_{\alpha/\beta}^{0,\beta}\psi(t_{n-1}))\biggr| = \biggl|k \frac{d^2}{d t^2} I_{\alpha/\beta}^{0,\beta}\psi(t^{*})\biggr|\\
      \le & k \int\limits_{0}^1\tau^{\frac{\alpha}{\beta}-1}(1-\tau^{\frac{\alpha}{\beta}})^{\beta-1} \tau^2 |\psi''(\tau t^*)|\, d \tau\le \frac{C }{n} t_n^{\alpha-1}.
      \end{split}
 \end{equation}
For the second term we repeat the same steps as in Theorem \ref{Thm:Order} but now use the fact that $\psi$ satisfies \eqref{EK:CondPsi}
 \begin{equation}
 \begin{split}
      |t_n\overline{\partial}I_{\alpha/\beta}^{0,\beta}\psi(t_n)   - t_n\overline{\partial}L_{\alpha/\beta}^{0,\beta}\psi(t_n)|\le & t_n \sum\limits_{i=1}^{n-1} \frac{\alpha}{\beta \Gamma(\beta)}\int\limits_{i-1}^{i}(f(n,s)-f(n-1,s))(i-s)|\psi'(\sigma_i k)|\, d s\\
    \ & + t_n |\psi'(x_n k)|\frac{\alpha}{\beta \Gamma(\beta)}\int\limits_{n-1}^{n}f(n,s)(n-s)\, d s\\
    = & t_n \sum\limits_{i=2}^{n-1} \frac{\alpha}{\beta \Gamma(\beta)}\int\limits_{i-1}^{i}(f(n,s)-f(n-1,s))(i-s)(\sigma_i k)^{\alpha-1}\, d s\\
    \ & + t_n |\psi'(x_n k)|\frac{\alpha}{\beta \Gamma(\beta)}\int\limits_{n-1}^{n}f(n,s)(n-s)\, d s\le  C t_n^\alpha \frac{1}{n^\beta}.
      \end{split}
 \end{equation}
Note also that when $\alpha>\beta$ the above difference can  be estimated differently
 \begin{equation}
     \begin{split}
       \frac{\beta \Gamma(\beta)}{\alpha}(t_n\overline{\partial}I_{\alpha/\beta}^{0,\beta}\psi(t_n)   - & t_n\overline{\partial}L_{\alpha/\beta}^{0,\beta}\psi(t_n))=    \frac{t_n}{k} \sum\limits_{i=2}^{n} \int\limits_{i-1}^{i}(f(n,s)-f(n-1,s-1))(\psi(s k)-\psi(i k))\, d s\\
       \ & + \frac{t_n}{k} \sum\limits_{i=2}^{n-1} \int\limits_{i-1}^{i}f(n-1,s)(\psi((s+1) k)-\psi((i+1) k)-(\psi(s k)-\psi(i k)))\, d s\\
       \ & +  \frac{t_n}{k} \int\limits_{0}^{1}f(n-1,s)(\psi((s+1) k)-\psi((i+1) k)-(\psi(s k)-\psi(i k)))\, d s\\
       \ & + \frac{t_n}{k}\int\limits_{0}^{1}(f(n,s))(\psi(s k)-\psi(i k))\ d s.
     \end{split}
 \end{equation}
 Performing straightforward calculation we obtain
 \begin{equation}
 \begin{split}
     \frac{t_n}{k} \sum\limits_{i=2}^{n} \int\limits_{i-1}^{i}(f(n,s)-f(n-1,s-1))|\psi(s k)- \psi(i k)|\, d s \le &   t_n k^{\alpha-1} \sum\limits_{i=2}^{n} \int\limits_{i-1}^{i}(f(n,s)-f(n-1,s-1))\, d s\\
     \le & C \frac{t_n^\alpha}{n^{\alpha/\beta+\alpha-1}},
     \end{split}
 \end{equation}
 and
 \begin{align}
 \frac{t_n}{k} \int\limits_{0}^{1}f(n-1,s)|\psi((s+1) k)-\psi((i+1) k)-(\psi(s k)-\psi(i k))|\, d s &\le C \frac{t_n^\alpha}{n^{\alpha/\beta+\alpha-1}},\\
    \frac{t_n}{k}\int\limits_{0}^{1}(f(n,s))|\psi(s k)-\psi(i k)|\ d s&\le C \frac{t_n^\alpha}{n^{\alpha/\beta+\alpha-1}}
    .
 \end{align}
 and
 \begin{equation}
    \frac{t_n}{k} \sum\limits_{i=2}^{n-1} \int\limits_{i-1}^{i}f(n-1,s)|\psi((s+1) k)-\psi((i+1) k)-(\psi(s k)-\psi(i k))|\, d s\le C \biggl(\frac{t_n}{n^{\alpha/\beta+\alpha-1}}+\frac{t_n^\alpha}{n}\biggr).
 \end{equation}
  To obtain appropriate discretization error estimates for the operator $K_{\alpha/\beta}^{0,\beta}$ we follow the same steps as above using the condition $t^{\alpha-2}|\psi''(t)|\le C$. The proof is complete.
 \end{proof}
\section{Stability}
\label{Results}
Using the proper discretization of the Erd\'elyi--Kober fractional derivative we can rewrite \eqref{Eq:WeakForm} in a semi-discrete form. To that end, let $U^n \in H^1(\mathbb{R})$ for all $n\in \mathbb{N}$, be the solution of the semi-discrete problem
\begin{equation}\label{TimeDiscreteEq}
   \begin{array}{l}
         (\overline{\partial}U^n,\chi)=-\frac{\alpha}{\beta}t_n^{\alpha-1} (G_{\alpha/\beta}^{\beta-1,1-\beta} U_{x}^n,\chi_x),\quad \forall \chi \in H^1(\mathbb{R}),
    \end{array}
\end{equation}
with $U^0=u(x,0)$.
Before we proceed to the main result of this section we will prove certain properties of the coefficients $c_{n,i}$ that appear in $L_{\alpha/\beta}^{\beta-1,1-\beta}$.
\begin{lem}
Fix $0<\alpha,\beta<1$ and $n>1$. For each $i\in \{1,2,\ldots,n-1\}$, we have
\begin{equation}\label{Lem:Cni}
    (n+\alpha) c_{n,i}-n c_{n-1,i}<0,
\end{equation}
where $c_{n,i}$ is defined as in \eqref{CoeffDef}.
\end{lem}
\begin{proof}
Let us note that using standard variable change for the integral in $c_{n,i}$ we obtain
\begin{equation}
    c_{n,i}=\frac{\alpha}{\beta \Gamma(\beta)}\int\limits_{\frac{i-1}{n}}^{\frac{i}{n}} \tau^{\frac{\alpha}{\beta}-1} (1-\tau^{\frac{\alpha}{\beta}})^{\beta-1}\, d \tau=\frac{\alpha}{\beta \Gamma(\beta)}\int\limits_{i-1}^{i}\frac{1}{n} \biggl(\frac{\tau}{n}\biggr)^{\frac{\alpha}{\beta}-1} \biggl(1-\biggl(\frac{\tau}{n}\biggr)^{\frac{\alpha}{\beta}}\biggr)^{\beta-1}\, d \tau.
\end{equation}
Then, we use above to rewrite the term $(n+\alpha)c_{n,i} - n c_{n-1,i}$ in a more tractable form
\begin{equation}
    \begin{split}
      (n+\alpha)c_{n,i} - n c_{n-1,i}= &  \frac{\alpha}{\beta \Gamma(\beta)}\int\limits_{i-1}^{i}\biggl[\frac{n+\alpha}{n} \biggl(\frac{\tau}{n}\biggr)^{\frac{\alpha}{\beta}-1} \biggl(1-\biggl(\frac{\tau}{n}\biggr)^{\frac{\alpha}{\beta}}\biggr)^{\beta-1}\\
      \ & -\frac{n}{n-1} \biggl(\frac{\tau}{n-1}\biggr)^{\frac{\alpha}{\beta}-1} \biggl(1-\biggl(\frac{\tau}{n-1}\biggr)^{\frac{\alpha}{\beta}}\biggr)^{\beta-1}\biggr]\, d \tau\\
      =&\frac{\alpha}{\beta \Gamma(\beta)}\int\limits_{i-1}^{i}\biggl[\biggl(1+\frac{\alpha}{n}\biggr) \bigl(\frac{\tau}{n}\bigr)^{\frac{\alpha}{\beta}-1} \biggl(1-\bigl(\frac{\tau}{n}\bigr)^{\frac{\alpha}{\beta}}\biggr)^{\beta-1}\\
      \ & -\biggl(1+\frac{1}{n-1}\biggr)\bigl(\frac{\tau}{n-1}\bigr)^{\frac{\alpha}{\beta}-1} \biggl(1-\bigl(\frac{\tau}{n-1}\bigr)^{\frac{\alpha}{\beta}}\biggr)^{\beta-1}\biggr]\, d \tau.
    \end{split}
\end{equation}
Next, for $\tau\in [0,n-1]$ we introduce the auxiliary function
\begin{equation}
    f(\tau,n)=\biggl(1+\frac{\alpha}{n}\biggr) \biggl(\frac{\tau}{n}\biggr)^{\frac{\alpha}{\beta}-1} \biggl(1-\biggl(\frac{\tau}{n}\biggr)^{\frac{\alpha}{\beta}}\biggr)^{\beta-1}-\biggl(1+\frac{1}{n-1}\biggr)\biggl(\frac{\tau}{n-1}\biggr)^{\frac{\alpha}{\beta}-1} \biggl(1-\biggl(\frac{\tau}{n-1}\biggr)^{\frac{\alpha}{\beta}}\biggr)^{\beta-1}.
\end{equation}
After carefull observation we note that $\lim _{\tau\to n-1}f(\tau,n)=-\infty$. Let us assume now that there exists some  $\tau^*>0$ that satisfies $f(\tau^*,n)=0$. Hence,
\begin{equation}
    \biggl(1+\frac{\alpha}{n}\biggr) \biggl(\frac{\tau^*}{n}\biggr)^{\frac{\alpha}{\beta}-1} \biggl(1-\biggl(\frac{\tau^*}{n}\biggr)^{\frac{\alpha}{\beta}}\biggr)^{\beta-1}=\biggl(1+\frac{1}{n-1}\biggr)\biggl(\frac{\tau^*}{n-1}\biggr)^{\frac{\alpha}{\beta}-1} \biggl(1-\biggl(\frac{\tau^*}{n-1}\biggr)^{\frac{\alpha}{\beta}}\biggr)^{\beta-1}
\end{equation}
or
\begin{equation}
    \frac{n^{\frac{\alpha}{\beta}}-\tau^{*\frac{\alpha}{\beta}}}{(n-1)^{\frac{\alpha}{\beta}}-\tau^{*\frac{\alpha}{\beta}}}=\biggl(\frac{n^{\alpha+1}}{(n-1)^\alpha (n+\alpha)}\biggr)^{\frac{1}{\beta-1}}.
\end{equation}
In the last equation, only the expression on left-hand side depends on $\tau^*$. We perform a straightforward calculation and obtain that the expression on the left is an increasing function with respect to $\tau^*$. Next, one can also notice that 
\begin{equation}
    g(n)\coloneqq \frac{n^{\frac{\alpha}{\beta}}}{(n-1)^{\frac{\alpha}{\beta}}}-\biggl(\frac{n^{\alpha+1}}{(n-1)^\alpha (n+\alpha)}\biggr)^{\frac{1}{\beta-1}},
\end{equation}
is decreasing function of $n$ and $\lim_{n\to \infty}g(n)=0$. Therefore, for each $n> 1$ we have $g(n)>0$ and thus there is no $\tau^*$ such that $f(\tau^*,n)=0$. Based on the above analysis we conclude \eqref{Lem:Cni}.
\end{proof}
Furthermore, in proving the convergence of the approximate solution to the exact one we will use the result presented in \cite{li2018} (Lemma 2) and originally proved in \cite{heywood1990}.
\begin{lem}[\cite{heywood1990}]\label{LemmaEstimate}
Let $\tau,\ B$ and $a_i,\ b_i, \ c_i, \ \gamma_i,$ for $i\in \mathbb{N}$, be nonnegative numbers such that
\begin{equation}
    a_n +\tau \sum\limits_{i=0}^n b_i\le \tau \sum\limits_{i=0}^n \gamma_i a_k +\tau \sum\limits_{i=0}^n c_i+B, \quad \mathrm{for}\ \ n\ge 0.
\end{equation}
Suppose that  $\tau \gamma_i<1$, for all $i$, and set $\sigma_i=(1-\tau \gamma_i)^{-1}$. Then
\begin{equation}
    a_n +\tau\sum\limits_{i=1}^n b_i\le (\tau \sum\limits_{i=1}^n c_i+B)\exp{(\tau \sum\limits_{i=1}^n\gamma_i \sigma_i)}.
\end{equation}
\end{lem}
Now, let us consider the stability of the equation \eqref{TimeDiscreteEq}
\begin{thm}
Suppose that the problem \eqref{Eq:WeakForm} has a unique solution for $0<\alpha,\beta<1$. Then the time-discrete problem \eqref{TimeDiscreteEq} has a unique solution $U^n$. Moreover, we have
\begin{equation}
\label{StabilityIneq}
    \|U^n\|\le 2^{1/2} \|U^0\|,\quad n\in\mathbb{N}.
\end{equation}
\end{thm}
\begin{proof}
From the standard theorem of the elliptic linear equations we have that for fixed $n$ \eqref{TimeDiscreteEq} possesses a unique solution. Next, let us choose $\chi=U^n$ in \eqref{TimeDiscreteEq} and  use
\begin{equation}
    (\overline{\partial}U^n,U^n)\ge \frac{1}{2} \overline{\partial}\|U^n\|^2,
\end{equation}
to obtain
\begin{equation}\label{eq1}
         \|U^n\|^2\le \|U^{n-1}\|^2  -2 k\frac{\alpha}{\beta}t_n^{\alpha-1} (G_{\alpha/\beta}^{\beta-1,1-\beta} U_{x}^n,\chi_x).
\end{equation}
We use the definition of the discrete operator $G_{\alpha/\beta}^{\beta-1,1-\beta}$ to rewrite the above equation in the following form
\begin{align*}
    \|U^n\|^2\le \|U^{n-1}\|^2  -2 k\frac{\alpha}{\beta}t_n^{\alpha-1} \biggl[ \beta (L_{\alpha/\beta}^{\beta-1,1-\beta} U_{x}^n,U_x^n) + \frac{\beta}{\alpha}\frac{t_n}{k}\biggl( (L_{\alpha/\beta}^{\beta-1,1-\beta} U_{x}^n,U_x^n)-(L_{\alpha/\beta}^{\beta-1,1-\beta} U_{x}^{n-1},U_x^n)\biggl)\biggl],
\end{align*}
or
\begin{align}\label{Stability:MainInEq}
    \|U^n\|^2\le \|U^{n-1}\|^2  -2 k\frac{\alpha}{\beta}t_n^{\alpha-1}  (\sum\limits_{i=1}^{n-1}(\beta c_{n,i}+\frac{\beta}{\alpha}\frac{t_n}{k} (c_{n,i}-c_{n-1,i})) U_{x}^i,U_x^n)- 2 k \frac{\alpha}{\beta} t_n^{\alpha-1} (\beta c_{n,n} +\frac{\beta}{\alpha}\frac{t_n}{k} c_{n,n})\|U_x^n\|^2.
\end{align}
In the next step we use \eqref{Lem:Cni} and get
\begin{equation}
\begin{aligned}
    \|U^n\|^2\le & \|U^{n-1}\|^2  +2 k\frac{\alpha}{\beta}t_n^{\alpha-1}  \sum\limits_{i=1}^{n-1}(-\beta c_{n,i}+\frac{\beta}{\alpha}\frac{t_n}{k} (c_{n-1,i}-c_{n,i}))\| U_{x}^i\|\cdot\|U_x^n\|\\
    \ & - 2 k \frac{\alpha}{\beta} t_n^{\alpha-1} (\beta c_{n,n} +\frac{\beta}{\alpha}\frac{t_n}{k} c_{n,n})\|U_x^n\|^2,
    \end{aligned}
    \end{equation}
    \begin{equation}
    \begin{aligned}
    \|U^n\|^2 \le & \|U^{n-1}\|^2  + k\frac{\alpha}{\beta}t_n^{\alpha-1}  \sum\limits_{i=1}^{n-1}(-\beta c_{n,i}+\frac{\beta}{\alpha}\frac{t_n}{k} (c_{n-1,i}-c_{n,i}))(\| U_{x}^i\|^2+\|U_x^n\|^2)- 2  t_n^{\alpha}  c_{n,n}\|U_x^n\|^2\\
    \ & -2\alpha k t_n^{\alpha-1} c_{n,n} \|U_x^n\|^2,
    \end{aligned}
    \end{equation}
    \begin{align}
    \|U^n\|^2\le & \|U^{n-1}\|^2-\alpha k t_n^{\alpha-1}\sum\limits_{i=1}^n c_{n,i} \|U_x^i\|^2 -t_n^{\alpha}\sum\limits_{i=1}^n c_{n,i} \|U_x^i\|^2 + t_n^{\alpha}\sum\limits_{i=1}^{n-1} c_{n-1,i} \|U_x^i\|^2 -\frac{\alpha k t_n^{\alpha-1}}{\beta \Gamma(\beta)} \|U_x^n\|^2,
    \end{align}
    \begin{equation}
        \begin{aligned}
           \|U^n\|^2 + \biggl(1+\frac{\alpha}{n}\biggr)t_n^\alpha\sum\limits_{i=1}^n c_{n,i} \|U_x^i\|^2 \le &  \|U^{n-1}\|^2  + t_{n-1}^{\alpha} \biggl(1+\frac{\alpha}{n-1}\biggr) \sum\limits_{i=1}^{n-1} c_{n-1,i}\| U_{x}^i\|^2 \\
    \ & + \biggl(t_n^\alpha-t_{n-1}^{\alpha} \biggl(1+\frac{\alpha}{n-1}\biggr)\biggr)  \sum\limits_{i=1}^{n-1} c_{n-1,i}\| U_{x}^i\|^2. 
    \label{Stability:SystemEq} 
        \end{aligned}
    \end{equation}
Let us define
\begin{equation}\label{Def:En}
    E^n\coloneqq\|U^n\|^2 + \biggl(1+\frac{\alpha}{n}\biggr) t_n^\alpha\sum\limits_{i=1}^n c_{n,i} \|U_x^i\|^2.
\end{equation}
Then, using the notation $E^n$ we rewrite the last inequality  \eqref{Stability:SystemEq} in following manner
\begin{equation}
    E^n\le E^{n-1}+\biggl(t_n^\alpha-t_{n-1}^{\alpha} \biggl(1+\frac{\alpha}{n-1}\biggr)\biggr)  \sum\limits_{i=1}^{n-1} c_{n-1,i}\| U_{x}^i\|^2 \le E^{n-1}.
\end{equation}
Finally, we have
\begin{equation}
    E^n\le E^{1}=\|U^1\|^2+\frac{(1+\alpha)}{ \Gamma(\beta+1)} t_1^\alpha\|U_x^1\|^2.
\end{equation}
Now, let us consider inequality \eqref{Stability:MainInEq} with $n=1$ separately
\begin{equation}
    \begin{split}
        \|U^1\|^2\le & \|U^{0}\|^2 -\frac{2\alpha k t_1^{\alpha-1}}{\beta \Gamma(\beta)} \|U_x^1\|^2-\frac{2 t_1^{\alpha}}{\beta\Gamma(\beta)}\|U^1\|^2.
    \end{split}
\end{equation}
From above inequality we immediately get $\|U^1\|^2\le \|U^0\|^2$ and $t_1^\alpha\|U_x^1\|^2/\Gamma(\beta+1)\le \|U^0\|^2/2$. Hence, using these inequalities we get
\begin{equation}
    E^n\le \biggl(1 + \frac{1+\alpha}{2}\biggr)\|U_0\|^2\le 2\|U_0\|^2. 
\end{equation}
From the above inequalities and \eqref{Def:En} we conclude \eqref{StabilityIneq}.
\end{proof}
We are ready to prove the theorem concerning convergence of the error for semi-discrete equation \eqref{TimeDiscreteEq} when the time step $k$ approaches zero.
\begin{thm}\label{Thm:EqOrder}
Suppose that the problem \eqref{Eq:WeakForm} has a unique solution $u(t,x)$ for $0<\alpha,\beta<1$, with $u_0(x)\in H^2(\mathbb{R})$, such that
\begin{align}
    \|u\|+t^{1-\alpha}\|u_t\|+t^{2-\alpha} \|u_{t t}\|\le & C,\\
     \|u_x\|+t^{1-\alpha}\|u_{tx}\|+t^{2-\alpha} \|u_{ttx}\|\le & C.
\end{align}
Then, there exist positive constant $D$ such that
 \begin{equation}
     \|U^n-u(t_n)\|\le D\left\{\begin{array}{cl}
          k^{\alpha-\mu/2}& \text{for}\quad \alpha\ne \beta,   \\
         k^{\alpha} ((-\log k)^{1/2}+k^{\mu/2})& \text{for}\quad \alpha=\beta,
     \end{array}\right.
 \end{equation}
 where $\mu>0$ is arbitrary small.
\end{thm}
\begin{proof}
We substitute $e^n=U^n-u(t_n)$ into \eqref{TimeDiscreteEq} to get
\begin{equation}
         (\overline{\partial}e^n,\chi)=-\frac{\alpha}{\beta}t_n^{\alpha-1} (G_{\alpha/\beta}^{\beta-1,1-\beta} e_{x}^n,\chi_x)+(\kappa_1,\chi)+\frac{\alpha}{\beta}t_n^{\alpha-1}(\kappa_2,\chi_x),
\end{equation}
where
\begin{align*}
    \kappa_1& =\overline{\partial}u(t_n)-u_t(t_n),\\
    \kappa_2 &= D_{\alpha/\beta}^{\beta-1,1-\beta}u_x(t_n)-G_{\alpha/\beta}^{\beta-1,1-\beta} u_x(t_n),
\end{align*}
and we make  use a fact that $u$ satisfies eq. \eqref{Eq:WeakForm}. 
For the clarity of the main result we assume that for $n=0$ we have $e^0=U^0-u(0)=0$. Let us now choose $\chi=e^n$ and obtain
\begin{equation}
   \frac{1}{2} \overline{\partial}\|e^n\|^2\le-\frac{\alpha}{\beta}t_n^{\alpha-1} (G_{\alpha/\beta}^{\beta-1,1-\beta} e_{x}^n,e_x^n)+(\kappa_1,e^n)+\frac{\alpha}{\beta}t_n^{\alpha-1}(\kappa_2,e_x^n).
\end{equation}
Next, we follow the same steps as before, in the stability analysis, to get 
\begin{equation}
    \begin{split}
        \|e^n\|^2 \le & \|e^{n-1}\|^2 + k\frac{\alpha}{\beta}t_n^{\alpha-1} \biggl(\sum\limits_{i=1}^{n-1} \biggl(\frac{\beta}{\alpha}\frac{t_n}{k}(c_{n-1,i}-c_{n,i})-\beta c_{n,i}\biggr) (\|e_x^i\|^2+ \|e_x^n\|^2)\biggr)\\
        \ & -2 k \frac{\alpha}{\beta} t_n^{\alpha-1} (\beta c_{n,n}+\frac{\beta}{\alpha}\frac{t_n}{k} c_{n,n}) \|e_x^n\|^2 +2 k(\kappa_1,e^n)+2 k\frac{\alpha}{\beta}t_n^{\alpha-1}(\kappa_2,e_x^n), \\
         \|e^n\|^2 + & t_n^\alpha\sum\limits_{i=1}^{n} c_{n,i} \|e_x^i\|^2\le  \|e^{n-1}\|^2 + t_n^{\alpha} \sum\limits_{i=1}^{n-1} c_{n-1,i} \|e_x^i\|^2-k \alpha t_n^{\alpha-1}\sum\limits_{i=1}^{n} c_{n,i} \|e_x^i\|^2\\
         \ & -k \frac{\alpha}{\beta\Gamma(\beta)} t_n^{\alpha-1} \|e_x^n\|^2
         +2 k(\kappa_1,e^n)+2 k\frac{\alpha}{\beta}t_n^{\alpha-1}(\kappa_2,e_x^n).
         \end{split}\label{FewIneq}
\end{equation}
Now, with a special care, let us investigate the components with $\kappa_1$ and $\kappa_2$. It is easy to see that
\begin{equation}
    \|\kappa_1\| \le\|\overline{\partial}u(t_n)-u_t(t_n)\|\le  \int\limits_{t_{n-1}}^{t_n}\|u_{t t}(t)\|\, d t,
 \end{equation}
 and thus 
 \begin{equation}
 \begin{split}
    k(\kappa_1,e^n)\le k\|\kappa_1\|\cdot \|e^n\|\le & \frac{1}{2}k t_n^{1-\mu} T\|\kappa_1\|^2+\frac{1}{2}\frac{k t_n^{\mu-1}}{T}\|e^n\|^2\le  C k^2  t_n^{1-\mu} \int\limits_{t_{n-1}}^{t_n}\|u_{t t}(t)\|^2\, d t +  \frac{1}{2}\frac{k t_n^{\mu-1}}{T}\|e^n\|^2\\
    \le & C k^2 t_n^{1-\mu} \int\limits_{t_{n-1}}^{t_n}t^{2\alpha-4}\, d t +  \frac{1}{2}\frac{k t_n^{\mu-1}}{T}\|e^n\|^2\\
    \le & C k^2 t_n^{1-\mu} \frac{(n k)^{2\alpha-3}-((n-1)k)^{2 \alpha-3}}{2 \alpha-3}+ \frac{1}{2} \frac{k t_n^{\mu-1}}{T}\|e^n\|^2\\
    \le & C k^{2\alpha-\mu}  n^{2\alpha-\mu-3} +  \frac{1}{2}\frac{k t_n^{\mu-1}}{T}\|e^n\|^2,
    \end{split}
 \end{equation}
 where again $\mu$ is some arbitrary small positive constant.
 For the component with $\kappa_2$ let us note that we can write
  \begin{equation}
 \begin{split}
     \frac{\alpha}{\beta}k t_n^{\alpha-1}(\kappa_2,e_x^n)= &  \alpha k t_n^{\alpha-1}( L_{\alpha/\beta}^{0,\beta} u_x(t_n)-I_{\alpha/\beta}^{0,\beta} u_x(t_n),e_x^n)\\
     \ & + k t_n^{\alpha-1}(\frac{t_n}{k}(L_{\alpha/\beta}^{0,\beta} u_x(t_n)-L_{\alpha/\beta}^{0,\beta} u_x(t_{n-1}))-t_n\frac{d}{d t} I_{\alpha/\beta}^{0,\beta} u_x(t_n),e_x^n).
     \end{split}
 \end{equation}
 Then, using Proposition \ref{PropSingular} we have
 \begin{equation}
 \begin{split}
      \alpha k t_n^{\alpha-1}( L_{\alpha/\beta}^{0,\beta} u_x(t_n)-I_{\alpha/\beta}^{0,\beta} u_x(t_n),e_x^n)\le & \alpha k t_n^{\alpha-1}\| L_{\alpha/\beta}^{0,\beta} u_x(t_n)-I_{\alpha/\beta}^{0,\beta} u_x(t_n)\|\cdot\|e_x^n\|\\
      \le & C k t_n^{\alpha-1} \frac{t_n^\alpha}{n^{\min\{\alpha/\beta+\alpha,1\}}} \cdot\|e_x^n\|\le C k^{2 \alpha} \frac{t_n^\alpha}{n^{\min\{1+2\alpha/\beta,3-2\alpha\}}}+ \varepsilon_1 k t_n^{\alpha-1}\|e_x^n\|^2 ,
      \end{split}
 \end{equation}
 and
 \begin{equation}
 \begin{split}
      k t_n^{\alpha-1} (t_n\overline{\partial} L_{\alpha/\beta}^{0,\beta} u_x(t_n) -  t_n \overline{\partial} I_{\alpha/\beta}^{0,\beta} u_x(t_n),e_x^n) \le &  k t_n^{\alpha-1} \|t_n\overline{\partial} L_{\alpha/\beta}^{0,\beta} u_x(t_n) -  t_n \overline{\partial} I_{\alpha/\beta}^{0,\beta} u_x(t_n)\|\cdot\|e_x^n\|\\
      \le & C k t_n^{\alpha-1} \frac{t_n^\alpha}{n^{\max \{\beta,\alpha/\beta+\alpha-1\}}} \cdot \|e_x^n\|\\
      \le & C k^{2\alpha} \frac{t_n^\alpha}{n^{1+2{\max \{\beta,\alpha/\beta+\alpha-1\}}-2\alpha}}+k t_n^{\alpha-1} \varepsilon_2 \|e_x^n\|^2, 
      \end{split}
 \end{equation}
 and
 \begin{equation}
 \begin{split}
    k t_n^{\alpha-1} ( t_n\frac{1}{k}(I_{\alpha/\beta}^{0,\beta} u_x(t_n)- I_{\alpha/\beta}^{0,\beta} u_x(t_{n-1}))-t_n\frac{d}{d t} I_{\alpha/\beta}^{0,\beta} u_x(t_{n})),e_x^n)\le & C k t_n^{\alpha-1} \frac{t_n^{2\alpha}}{n^2} + \varepsilon_3 k t_n^{\alpha-1}\|e_x^n\|^2\\
    \le & C k^{2\alpha}  \frac{t_n^{\alpha}}{n^{3-2\alpha}} + \varepsilon_3 k t_n^{\alpha-1}\|e_x^n\|^2.
    \end{split}
 \end{equation}
 Next, similarly to the stability analysis we substitute
 \begin{equation}
     F^n=\|e^n\|^2 + \biggl(1+\frac{\alpha}{n}\biggr)t_n^\alpha\sum\limits_{i=1}^{n} c_{n,i} \|e_x^n\|^2,
 \end{equation}
 into last inequality in \eqref{FewIneq}, and obtain
 \begin{equation}
     \begin{split}
        F^n\le& F^{n-1} + \biggl(t_n^{\alpha}-\biggl(1+\frac{\alpha}{n-1}\biggr)t_{n-1}^{\alpha}\biggr) \sum\limits_{i=1}^{n-1} c_{n-1,i} \|e_x^i\|^2\\
         \ & -k \frac{\alpha}{\beta\Gamma(\beta)} t_n^{\alpha-1} \|e_x^n\|^2
         +2 k(\kappa_1,e^n)+2 k\frac{\alpha}{\beta}t_n^{\alpha-1}(\kappa_2,e_x^n)\\
         \le& F^{n-1} 
          -k \frac{\alpha}{\beta\Gamma(\beta)} t_n^{\alpha-1} \|e_x^n\|^2
          +C k^{2\alpha-\mu}  \frac{1}{n^{3-2\alpha+\mu}} + \frac{k t_n^{\mu-1}}{T}\|e^n\|^2 +C k^{2 \alpha} \frac{t_n^\alpha}{n^{\min\{1+2\alpha/\beta,3-2\alpha\}}}\\
          \ & + \varepsilon_1 k t_n^{\alpha-1}\|e_x^n\|^2 +C k^{2\alpha} \frac{t_n^\alpha}{n^{1+{2\max \{\beta,\alpha/\beta+\alpha-1\}}-2\alpha}}+k t_n^{\alpha-1} \varepsilon_2 \|e_x^n\|^2+C k^{2\alpha}  \frac{t_n^{\alpha}}{n^{3-2\alpha}} + \varepsilon_3 k t_n^{\alpha-1}\|e_x^n\|^2,
     \end{split}\label{FIneq}
 \end{equation}
 where we used that fact that for arbitrary $n\ge 2$ and $\alpha\in (0,1)$ we have
 \begin{equation}
   n^\alpha -(n-1)^\alpha \biggl(1+\frac{\alpha}{n-1}\biggr)<0.
 \end{equation}
 Now, because $\varepsilon_i,\ i=1,2,3$, are arbitrary positive constant we can assume that they satisfy
  \begin{equation}
     \varepsilon_1+ \varepsilon_2 +  \varepsilon_3 -\frac{\alpha}{\beta\Gamma(\beta)}<0.
 \end{equation}
 Therefore, using above we rewrite inequality \eqref{FIneq} in the following more compact form
 \begin{equation}
     \begin{split}
        F^n\le&   F^{n-1}+
           \frac{k t_n^{\mu-1}}{T} F^n+
          C k^{2\alpha-\mu}  \frac{1}{n^{3-2\alpha+\mu}}+C k^{2\alpha}\frac{1}{n^{1+2\alpha/\beta}}   +C k^{2\alpha} \frac{1}{n^{1+{2\max \{\beta,\alpha/\beta+\alpha-1\}}-2\alpha}}+C k^{2\alpha}  \frac{1}{n^{3-2\alpha}}.
     \end{split}
 \end{equation}
 Summing over the above formula from $2$ to $n$ we get
 \begin{equation}
   \begin{split}\label{Ineq:F_Convergence}
        F^n\le&   F^1+k\sum\limits_{i=2}^{n} \frac{t_n^{\mu-1}}{T}F^i+C k^{2\alpha-\mu}\sum\limits_{i=2}^{n} \biggl[\frac{1}{n^{3-2\alpha}}+\frac{1}{n^{1+2\alpha/\beta}}+\frac{1}{n^{1+2\max \{\beta,\alpha/\beta+\alpha-1\}-2\alpha}}\biggr].
     \end{split}  
 \end{equation}
 Let us now separately consider the first inequality in \eqref{FewIneq}  with $n=1$,
 \begin{equation}
     \begin{split}
     \|e^1\|^2 \le & \|e^{0}\|^2 
         -2 k \frac{\alpha}{\beta} t_1^{\alpha-1} (\beta c_{1,1}+\frac{\beta}{\alpha}\frac{t_1}{k} c_{1,1}) \|e_x^1\|^2 +2 k(\kappa_1,e^1)+2 k\frac{\alpha}{\beta}t_1^{\alpha-1}(\kappa_2,e_x^1)\\
         \le & \|e^{0}\|^2 
         -2 t_1^{\alpha} (\alpha +1)c_{1,1} \|e_x^1\|^2 +2 k(\kappa_1,e^1)+2 k\frac{\alpha}{\beta}t_1^{\alpha-1}(\kappa_2,e_x^1).
         \end{split}
         \end{equation}
         Note that we can estimate the terms with $\kappa_1$ and the differential part in $\kappa_2$ in the following way
         \begin{equation}
         \begin{split}
             k(\kappa_1,e^1)= & ( u(t_1)-u(0)-k u_t,e^1)\le \|\int\limits_{0}^k t u_{t t}\, d t\| \cdot \|e^1\|\le \int\limits_{0}^k t \|u_{t t}\|\, d t \cdot \|e^1\|\\
             \le & \frac{1}{4 \varepsilon^*}\biggl(\int\limits_{0}^k t\| u_{t t}\|\, d t\biggr)^2+ \varepsilon^* \|e^1\|^2\le \frac{1}{4 \varepsilon^*} \frac{1}{\alpha^2}k^{2 \alpha}+\varepsilon^* \|e^1\|^2.  
             \end{split}
         \end{equation}
         and
         \begin{equation}
         \begin{aligned}
               t_1^{\alpha}(I^{0,\beta}_{\alpha/\beta}u_{x}(t_1)-I^{0,\beta}_{\alpha/\beta}u_{x}(0)- k \frac{d}{d t}I^{0,\beta}_{\alpha/\beta}u_{x}(t_1),e_x^1)\le & C k^{\alpha} \biggl(\int\limits_{0}^k t \|\frac{d^2}{d t^2}I^{0,\beta}_{\alpha/\beta}u_{x}(t)\| \, d t\biggr)^2+   k t_1^{\alpha-1}\varepsilon_3 \|e_x^1\|^2\\
               \le & C k^{2\alpha} +  k t_1^{\alpha-1}\varepsilon_1\|e_x^1\|^2,
               \end{aligned}
               \end{equation}
               \begin{equation}
                  \begin{aligned}
            \frac{\alpha}{\Gamma(\beta+1)} k t_1^{\alpha-1} (u_{x}(t_1)-\Gamma(\beta+1)I^{0,\beta}_{\alpha/\beta}u_{x}(t_1),e_x^1)&\le C t_1^{3\alpha}+ \varepsilon_2 t_1^\alpha \|e_x^1\|^2,
            \end{aligned}
            \end{equation}
            \begin{equation}
                \begin{aligned}
            t_1^\alpha (c_{1,1}u_x(t_1)-I_{\alpha/\beta}^{0,\beta} u_x(t_1)+I_{\alpha/\beta}^{0,\beta} u_x(0),e_x^1)\le & C t_1^{3\alpha} + C t_1^\alpha \|u_{0,x}\|^2 + \varepsilon_1 \|e_x^1\|^2 \le C t_1^{3\alpha} + C t_1^\alpha + \varepsilon_1 \|e_x^1\|^2,
             \end{aligned}
         \end{equation}
         which together gives
         \begin{equation}
         \begin{split}
             k t_1^{\alpha-1}(\kappa_2,e_x^1) \le& C( k^{3\alpha} + k^{2\alpha}+k^\alpha) +   k^\alpha(\varepsilon_1+\varepsilon_2+\varepsilon_3) \|e_x^1\|^2,
            \end{split}
         \end{equation}
         and
         \begin{equation}
             k(\kappa_1,e^1)\le \frac{C}{\varepsilon^*}k^{2\alpha} +\varepsilon^*\|e^1\|^2,
         \end{equation}
         where $\varepsilon^*$ is some positive constant satisfying $2\varepsilon^*<1$  Hence, thanks to above estimations we get
         \begin{equation}
             \begin{split}
        \|e^1\|^2 + (1+\alpha)t_1^\alpha c_{1,1} \|e_x^1\|^2 \le &  -k \frac{\alpha}{\beta\Gamma(\beta)} t_1^{\alpha-1} \|e_x^1\|^2
         +C k^{2\alpha}+   \varepsilon^* \|e^1\|^2 +  C k^{3\alpha}\\
         \ & +C k^{2\alpha}+C k^\alpha+k^\alpha(\varepsilon_1+\varepsilon_2+t_1\varepsilon_3) \|e_x^1\|^2\\
          \le & C k^{\alpha} +C k^{2\alpha}+   \varepsilon^* \|e^1\|^2,
     \end{split}
 \end{equation}
 where we again choose $\varepsilon_1,\, \varepsilon_2, \, \varepsilon_3$ such that $\varepsilon_1+\varepsilon_2+\varepsilon_3-\alpha/\beta\Gamma(\beta)<0$.
 Hence, we get
 \begin{equation}
 \begin{split}
     F^1\le & \frac{C}{1-2 \varepsilon^*} k^{\alpha}.
     \end{split}
 \end{equation}
 what immediately implies $\|e^1\|\le C k^{\alpha/2}$. Let us note now that from the Sobolev Embedding Theory  we have $u(t)\in H^1(\mathbb{R})\subset C^{0,1/2}(\mathbb{R})$.  Moreover, $u(0)=u_{0}\in H^2(\mathbb{R})$ with $\lim_{x\to \pm \infty} |u_{0,x}(x)|=0$. Thus we have
 \begin{equation}
     |(u_x(0),e_x^1)|=|(u_{x x}(0),e^1)|\le \|u_{x x}(0)\|\cdot\|e^1\|\le C k^{\alpha/2},
 \end{equation}
 and thus
\begin{equation}
         \begin{split}
             k^{\alpha}(\kappa_2,e_x^1) \le& C( k^{3\alpha} + k^{2\alpha}+k^{3\alpha/2}) +   k^\alpha(\varepsilon_1+\varepsilon_2+\varepsilon_3) \|e_x^1\|^2.
            \end{split}
         \end{equation}
         Choosing again appropriate values for $\varepsilon_i,\ i\in \{1,2,3\}$, we get $F^1\le C k^{3\alpha/2}$. Repeating above procedure finitely many times we get $F^1\le C k^{2\alpha-\mu}$, where $\mu$ is arbitrarily small.
 Going back to inequality \eqref{Ineq:F_Convergence} for $n\ge 2$, we use Lemma \ref{LemmaEstimate} with $\gamma_i=t_i^{\mu-1}/T$ to obtain
  \begin{equation}
     \begin{split}
        F^n\le&   C k^{2\alpha-\mu}\sum\limits_{i=2}^{n} \biggl[\frac{1}{n^{3-2\alpha}}+\frac{1}{n^{1+2\alpha/\beta}}+\frac{1}{n^{1+2\max \{\beta,\alpha/\beta+\alpha-1\}-2\alpha}}\biggr]
        + C F^1.
     \end{split}
 \end{equation}
 Finally, using the asymptotic behaviour of the following series
 \begin{equation}
     \sum\limits_{i=1}^n\frac{1}{i^{2\nu+1-2 \alpha}}=\zeta{(1+2\nu-2 \alpha)}+\frac{1}{(2\alpha-2\nu)}\frac{1}{n^{2\nu-2\alpha}}+\frac{1}{2 n^{1+2\nu-2\alpha}}+\mathcal{O}\biggl(\frac{1}{n^{2+2\nu-2\alpha}}\biggr),\quad \mathrm{as}\quad n\to \infty,
 \end{equation}
 for $\nu\ne \alpha$, where $\zeta(x)$ is a Riemann zeta function, and
 \begin{equation}
    \sum\limits_{i=1}^n\frac{1}{i^{1}}= \biggl(\gamma -\log \biggl(\frac{1}{n}\biggr)\biggr)+\frac{1}{2 n} + \mathcal{O}\biggl(\frac{1}{n^2}\biggr),\quad \mathrm{as}\quad n\to \infty,
 \end{equation}
 for $\nu=\alpha$, where $\gamma$ is a Euler's constant, we get 
 \begin{equation}
     F^n \le \left\{\begin{array}{cl}
          C k^{2\alpha}+F^1 & \text{for}\quad \alpha\ne \beta,   \\
         C k^{2\alpha} (-\log k)+F^1 & \text{for}\quad \alpha=\beta.
     \end{array}\right.
 \end{equation}
 Hence, by the definition of $F^n$ we obtain the desired inequality
 \begin{equation}
 \begin{split}
     \|e^n\|\le \left\{\begin{array}{cl}
          C k^{\alpha}+(F^1)^{1/2} & \text{for}\quad \alpha\ne \beta,   \\
         C k^{\alpha} (-\log k)^{1/2}+(F^1)^{1/2} & \text{for}\quad \alpha=\beta.
     \end{array}\right.
     \end{split}
 \end{equation}
The proof is complete.
\end{proof}

\section{Numerical analysis}
\label{NumAnalysis}
In this section, we gather some numerical examples that support our previous results concerning the discretization of the Erd\'elyi--Kober differential and integration operators altogether with the numerical scheme for solving \eqref{Eq:WeakForm}.

In the beginning, let us consider the discrete operators $ L_{\alpha/\beta}^{0,\beta},\ K_{\alpha/\beta}^{\beta-1,1-\beta}$ $G_{\alpha/\beta}^{\beta-1,1-\beta}$. In the numerical demonstration of the discretization error the test functions $\phi(x)=x^{3/2}$ for the integral operator and $\psi(x) =x^2$ for differential operators were used. We chose these functions because, in general, for power functions one can easily calculate the exact value of analysed Erd\'elyi--Kober operators. In the Fig. \ref{fig:EK_Convergence} the discretization errors for the considered operators are depicted. In the doubly logarithmic scale, all graphs are straight lines as functions of the number of interval subdivisions. Therefore, it is reasonable to claim that this data support Theorem \ref{Thm:Order}.
\begin{figure}
    \centering
    \includegraphics{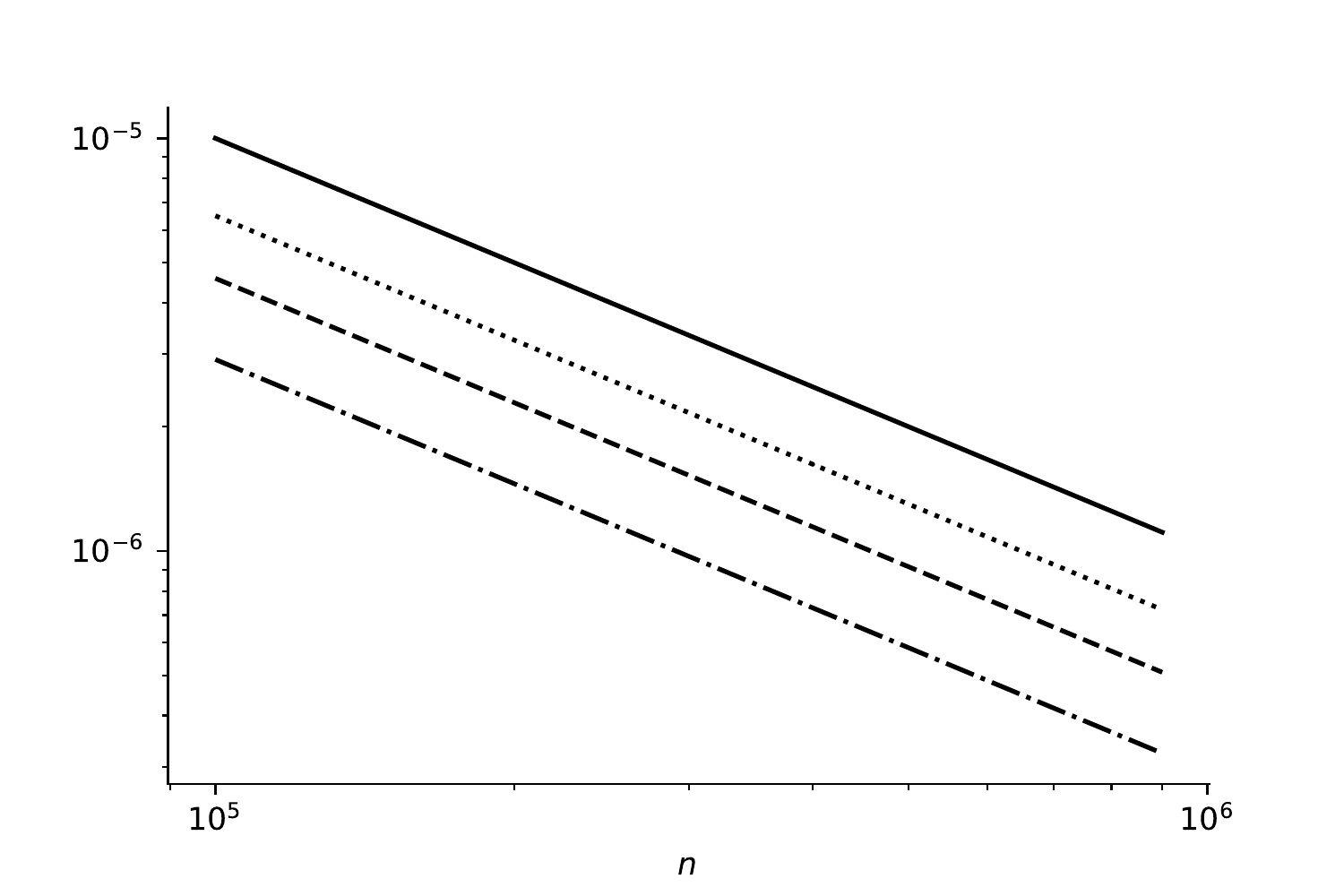}
    \caption{Numerical demonstration of the convergence order of the discretization error with respect to the number of divisions of the integration interval, depicted in the log-log scale. The solid line represent the reference graph of first order convergence. Errors related to $ L_{\alpha/\beta}^{0,\beta},\ K_{\alpha/\beta}^{\beta-1,1-\beta}$ $G_{\alpha/\beta}^{\beta-1,1-\beta}$ are presented as dashed, dotted and dash-dotted line, respectively. In this analysis the following parameters values were used: $\alpha=0.25,\ \beta=0.65$ and $x=1$.}
    \label{fig:EK_Convergence}
\end{figure}

To estimate the order of convergence for the other set of parameters, we use Aitken's extrapolation method (see \cite{linz1985}). According to it, the formula for order $p$ for the discrete operator $L_{\alpha/\beta}^{0,\beta}$, at point $t_n=1$, is given by
\begin{equation}
    p\approx \log_2\frac{\mathcal{A}_{2n}-\mathcal{A}_n}{\mathcal{A}_{4n}-\mathcal{A}_{2n}}, \quad \text{where} \quad \mathcal{A}_n := L_{\alpha/\beta}^{0,\beta}\phi(1) \text{ for } n \text{ grid points}. 
\end{equation}
Other operators are analysed analogously. This method of estimating the order of convergence was also used in \cite{plociniczak2017} where different methods of discretization of the Erd\'elyi--Kober fractional integral operator were investigated. Here, as a test function we use $\phi(t)=e^{t}$ and evaluate the error at $t=t_n=1$. Obtained results are presented in Table \ref{TableAitken}. It is easy to notice that all numerically determined orders of convergence are close to $1$, and moreover, when the number of partitions of interval $[0,1]$ increases, then the order $p$ is asymptotically approaching $1$. Hence, also for this method, the numerical results are in accordance with Theorem \ref{Thm:Order}.

\begin{table}
\begin{center}
{\renewcommand{\arraystretch}{1.8}
\begin{tabular}{ p{3cm} p{1.5cm} p{1.5cm}| p{1.5cm} p{1.5cm}}%
 \toprule
 \ & \multicolumn{2}{c}{$\alpha=0.7,\ \beta=0.15$} & \multicolumn{2}{c}{$\alpha=0.25,\ \beta=0.65$}\\
 \midrule
 & $n=10^5$   & $n=10^6$    & $n=10^5$   & $n=10^6$ \\
 \midrule
$L_{\alpha/\beta}^{0,\beta}\phi(x)$ & $0.97$   & $0.98$   &$1.0009$ &   $1.0004$\\
 $K_{\alpha/\beta}^{\beta-1,1-\beta} \phi(x)$& $0.90$   & $0.94$    & $0.9994$&   $0.9998$\\
 $G_{\alpha/\beta}^{\beta-1,1-\beta}\phi(x)$ & $0.88$   & $0.93$    & $1.0004$&   $1.0003$\\
 \bottomrule
\end{tabular}\label{TableAitken}
}
\end{center}
\caption{Estimated order of discretization errors.}
\end{table}

Now, we proceed to the analysis of the numerical scheme used to solve \eqref{Eq:WeakForm}. We discretize the first derivative with respect to time and Erd\'elyi--Kober fractional differential operator as in \eqref{TimeDiscreteEq}. Next, to approximate the exact solution in the spatial dimension, we use Galerkin--Hermite method. We use Hermite functions due to their rapid decay at the infinity and orthogonality in $L^2(\mathbb{R})$. To this end, following \cite{shen2011}, let us define Hermite polynomial $H_n(x),\ n\ge 0$. Using the Rodrigues' formula
\begin{equation}
    H_n(x)\coloneqq (-1)^n e^{x^2}\frac{d^n}{d x^n}e^{-x^2}, \quad x\in \mathbb{R}.
\end{equation}
Hermite polynomials are orthogonal in $L^2(\mathbb{R})$ with respect to the weight $w(x)=e^{-x^2}$. However, in our numerical scheme it is more appropriate to use Hermite functions $\widehat{H}_n(x)$
\begin{equation}
    \widehat{H}_n(x)=\pi^{-\frac{1}{4}} (2^n n!)^{-\frac{1}{2}} e^{-\frac{x^2}{2}} H_n(x).
\end{equation}
Let us notice the Hermite functions vanish exponentially as $x\to \pm \infty$ and are orthonormal in $L^2(\mathbb{R})$, that is
\begin{equation}
    \int_{-\infty}^{\infty} \widehat{H}_n(x)\widehat{H}_m(x)\, d x=\delta_{n m},
\end{equation}
and the inner product of the first-order derivatives satisfies \cite{shen2011},
\begin{equation}
    \int_{-\infty}^{\infty} \widehat{H}'_n(x)\widehat{H}'_m(x)\, d x=\widehat{\delta}_{n m}=\left\{
    \begin{array}{cc}
        -\frac{\sqrt{n(n-1)}}{2}, & m=n-2, \\
        n+\frac{1}{2}, &  n=m,\\
        -\frac{\sqrt{(n+2)(n+1)}}{2}, & m = n+2,\\
        0 , & \mathrm{otherwise},
    \end{array}\right.
\end{equation}
therefore the stiffness matrix is banded. Let $P_N$ be the space of polynomials of degree at most $N$ and 
\begin{equation}
    \widehat{P}_N\coloneqq \{ \widehat{H}_n(x)=\pi^{-\frac{1}{4}} (2^n n!)^{-\frac{1}{2}} e^{-\frac{x^2}{2}} H_n(x),\ H_n(x)\in P_N\}.
\end{equation}
Finally, we can rewrite the time-discrete problem \eqref{TimeDiscreteEq} in the following fully discrete form 
\begin{equation}\label{FullDiscreteEq}
   \biggl\{ \begin{array}{l}
          \mathrm{Find}\ U^n_N=U_N(t_n)\in \widehat{P}_N,  \quad \mathrm{such\ that}, \\
         (\overline{\partial}U^n_{N},\chi_N)=-\frac{\alpha}{\beta}t_n^{\alpha-1} (G_{\alpha/\beta}^{\beta-1,1-\beta} U^n_{N,x},\chi_{N,x}),\quad \forall \chi \in \widehat{P}_N,
    \end{array}\biggr.
\end{equation}
with $U^0_N=\widehat{\Pi}_N u(0,x)$, where $\widehat{\Pi}_N : L^2(\mathbb{R})\rightarrow \widehat{P}_N$ is orthogonal projection defined as \cite{shen2011}
\begin{equation}
    \widehat{\Pi}_N \coloneqq e^{-\frac{x^2}{2}}\Pi_N\bigl(u e^{\frac{x^2}{2}}\bigr),\quad u\in L^2(\mathbb{R}),
\end{equation}
and
\begin{equation}
    (u-\Pi_n u,v_n)_w=0,\quad \forall v_n\in P_n.
\end{equation}
It is clear that 
\begin{equation}
    \widehat{\Pi}_N u(x)=\sum\limits_{i=0}^N \gamma_i \widehat{H}_i, \quad \mathrm{where}\quad \gamma_i=\int\limits_{-\infty}^{+\infty} u(x) \widehat{H}_i(x)\, d x.
\end{equation}
Hence, by setting 
\begin{equation}
    U^n_N=\sum\limits_{i=0}^N\gamma_{n,i} \widehat{H}_i(x),
\end{equation}
and rewriting \eqref{FullDiscreteEq} $N+1$ times, where for the $k$th equation we choose $\chi_N = \widehat{H}_k(x)$, we can derive the system of algebraic equations
\begin{equation}
     A \widehat{\gamma}_n = B,\quad .
\end{equation}
where $\widehat{\gamma}_n=(\gamma_{n,0},\gamma_{n,1},\ldots,\gamma_{n,N})$ and
\begin{equation}
    A=[(\alpha k t_n^{\alpha-1}+t_n^{\alpha})c_{n,n}\widehat{\delta}_{i j}+\delta_{i j}],\quad B=[-t_n^{\alpha-1}\sum\limits_{l=1}^{n-1} (\alpha k c_{n,l}+ t_n(c_{n,l}-c_{n-1,l}))\langle \widehat{\gamma}_l,\widehat{\delta}_i\rangle]^T, 
\end{equation}
with $\widehat{\delta}_i=(\widehat{\delta}_{i 0},\widehat{\delta}_{i 1},\ldots,\widehat{\delta}_{i N})$, $i,j\in\{0,1,\ldots,N\}$ and $\langle\cdot,\cdot\rangle$ is a standard Euclidean dot product. Note, that to determine $\widehat{\gamma}_n$ we need to know the values for all  coefficients in previous time steps, i.e. $\widehat{\gamma}_k$, $k\in\{1,2,\ldots,n-1\}$. 

Next, let us notice that if we choose the initial condition $u(0,x)=e^{-x^2/2}$ we obtain $\widehat{\Pi}_N u(0,x)= \pi^{1/4}\widehat{H}_0$. Moreover, in the literature  devoted to the Erd\'elyi--Kober diffusion equation (see \cite{mura2008, mura2009, mura2008character, mura2008non,pagnini2012}), we can find the formula for the Green function corresponding to the eq. \eqref{Eq:Main} and \eqref{Eq:Mura},
\begin{equation}
    \mathcal{G}(t,x)=\frac{1}{2}\frac{1}{t^{\alpha/2}}M_{\beta/2}\biggl(\frac{|x|}{t^{\alpha/2}}\biggr),
\end{equation}
where $M_\mu(z)$ is the Mainardi function, also known as $M$-Wright function (for more details concerning Mainardi function see \cite{mainardi2010,gorenflo2007, gorenflo2000, mainardi2010Mura}), defined as 
\begin{equation}
\begin{split}
    M_\mu (z)=& \sum\limits_{n=0}^{\infty}\frac{(-z)^n}{n! \Gamma(-\mu n+(1+\mu)}\\
    = & \frac{1}{\pi}\sum\limits_{n=1}^{\infty}\frac{(-z)^{n-1}}{(n-1)!}\Gamma(\mu n) \sin (\pi \mu n),\quad 0<\mu<1.
    \end{split}
\end{equation}
For the special case, $\mu=1/2$, the Mainardi function can be rewritten in a more familiar form,
\begin{equation}
\label{Mainardi:Special}
    M_{1/2}(z)=\frac{1}{\sqrt{\pi}}e^{-z^2/4}.
\end{equation}
Using above for $\beta=1$ and $u(0,x)=\widehat{H}_0$ we are able to calculate the exact solution of \eqref{Eq:Main}, i.e.
\begin{equation}\label{ExactSolution}
    u(t,x)=\int\limits_{-\infty}^{\infty}\mathcal{G}(t,x-y) u(0,y)\, d y= \pi^{1/4}\frac{t^{-\frac{\alpha }{2}} e^{-\frac{x^2}{4 t^{\alpha }+2}}}{\sqrt{t^{-\alpha }+2}}.
\end{equation}
The above expression will be utilised to compare the analytical order of the method proved in Theorem \ref{Thm:EqOrder} with the data obtained in the numerical experiment. In Fig. \ref{fig:ErrorAnalysisLogLog} the errors of approximating the exact solution \eqref{ExactSolution} by $U_N^n$ for different values of $k$ are depicted in the doubly logarithmic scale. The solid line represents the power-law dependence derived analytically in Theorem \ref{Thm:EqOrder}. Using linear regression one can conclude that points representing the difference $\|\widehat{\Pi}_N u-U_N^n\|$, for different values of $n$, where $k=1/n$, are arranged in a line with a slope $0.58$. Note that the estimated order of convergence is equal to $\alpha=0.65$. This slight discrepancy is probably due to the slow temporal convergence of the method requiring a very small time step to fully resolve the error. However, we can conclude that the numerical estimate of the convergence error is in the right ballpark.

\begin{figure}
    \centering
     \includegraphics[width=0.7\textwidth]{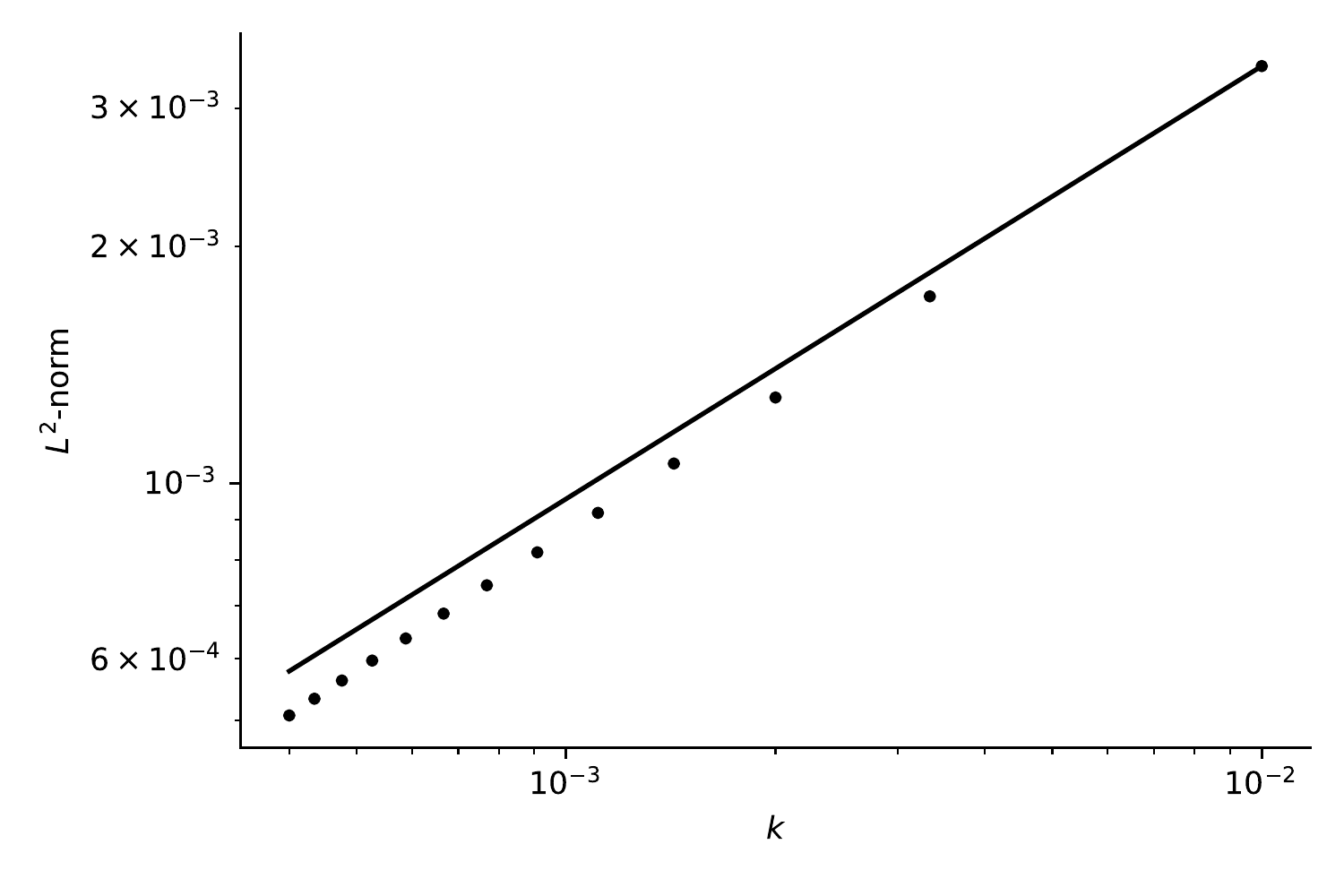}
    \caption{Numerical demonstration of order of convergence of $\|\widehat{\Pi}_N u-U_N^n\|$ (dots) with respect to time step $k$, depicted in log-log scale, where $u(t,x)$ is given by \eqref{ExactSolution} and $U_N^n$ is a solution of \eqref{FullDiscreteEq}.  The following parameters were used: $\alpha=0.65,\ \beta=1,\ t_n=1,\ N=30$. Solid line represent the linear function with slope equal $\alpha$.}
    \label{fig:ErrorAnalysisLogLog}
\end{figure}

To further investigate the error of approximation, we again can use the Aitken's method. In this case, to estimate the order of convergence with respect to the time step $k=t_n/n$ and for a fixed $N$ we calculate the ratio of the $L^2(\mathbb{R})$ norm of the appropriate difference, i.e.
\begin{equation}
    p\approx \log_2\frac{\|\mathcal{U}^n-\mathcal{U}^{2 n}\|}{\|\mathcal{U}^{2 n}-\mathcal{U}^{4 n}\|},
\end{equation}
where $\mathcal{U}^n\coloneqq U_N^n$ is a solution of \eqref{FullDiscreteEq} with $k=1/n$ and $t_n=1$. We use two initial condition functions: $u(0,x)=e^{-x^2/2}$ and $u(0,x)=1/(1+x^2)$ and compute the error at $t_n=1$. For the former, we use $N=5$ as the highest order of Hermite polynomial in orthogonal expansion, whereas for the latter choice of the initial condition we use $N=20$ to approximate the exact solution more accurately since its Hermite expansion is infinite. To calculate coefficients $\gamma_0$ for the latter initial condition, we use Gaussian quadrature rule. Obtained results are presented in Table \ref{TableAitkenEq}. For parameters $\alpha=0.7,\ \beta=0.15$ and both choices of the initial condition, we can see that the order of convergence quickly stabilizes near $\alpha$. However, in the case of smaller $\alpha$ obtained order estimates are in its slightly larger neighbourhood. In other words, the estimated order of convergence for a small value of $\alpha$ attains the analytically derived order more slowly. Nevertheless, in the case $\alpha<\beta$ the results presented in Table \ref{TableAitkenEq} agree with the order of convergence derived in Theorem \ref{Thm:EqOrder}.

\begin{table}
\begin{center}
{\renewcommand{\arraystretch}{1.8}
\begin{tabular}{ p{3cm} p{1.7cm} p{1.7cm}| p{1.7cm} p{1.7cm}}%
 \toprule
 \ & \multicolumn{2}{c}{$\alpha=0.7,\ \beta=0.15$} & \multicolumn{2}{c}{$\alpha=0.35,\ \beta=1$}\\
 \midrule
$u(0,x)$ & $n=500$   & $n=1000$    & $n=1000$   & $n=1500$ \\
 \midrule
$e^{-x^2/2}$ & \multicolumn{1}{c}{$0.645$}   & \multicolumn{1}{c|}{$0.652$}    &\multicolumn{1}{c}{$0.499$} &   \multicolumn{1}{c}{$0.435$}\\
 $\frac{1}{1+x^2} $& \multicolumn{1}{c}{$0.6878$}   & \multicolumn{1}{c|}{$0.6891$}& \multicolumn{1}{c}{$0.377$}&   \multicolumn{1}{c}{$0.301$}\\
 \bottomrule
\end{tabular}\label{TableAitkenEq}
}
\end{center}
\caption{Estimated temporal orders of the numerical scheme error calculated using Aitken's method.}
\end{table}

For estimation of the spatial discretization error, we again use the Gaussian initial condition $u(0,x)=e^{-x^2/2}$. The error $\| u(t_n,x)-U_N^n\|$ where $U_N^0=\Pi_N u(0,x)=\pi^{1/4} \widehat{H}(x)$ as a function of $N\in\{2,4,\ldots,24\}$ is depicted in Fig. \ref{fig:EK_Convergence_N}. For numerical convenience, we consider the error in $L^2(\mathbb{R})$ norm at point $t_n=1$ with $n=2000$. Let us notice that for the initial condition of the form $e^{x^2/2}$ with $ \beta=1$, the value of error in $L^2(\mathbb{R})$ norm can be easily obtained thanks to the simple formula for the Mainardi function \eqref{Mainardi:Special} and the fact that in this special case we are able to derive the exact form of the solution \eqref{ExactSolution}. For fixed $n$ the error decreases linearly on the semi-logarithmic scale indicating that the convergence is indeed spectral (exponential). After $N=20$ the error hardly changes saturating to the value determined by the temporal error. Increasing the number of time interval subdivisions $n$ would make this limit arbitrarily small. 
To make the error estimation independent on the temporal accuracy, instead of comparison with the exact solution, we can use a reference one. That is to say, we fix $n$ and numerically compute the solution for significantly larger number of Hermite terms in the expansion, say $N=60$. Then, for increasing number of $N$ we calculate the error $\|U_N^n-U_{60}^n\|$. Numerical results are presented in Fig. \ref{fig:EK_Convergence_N_num}. As we can see, the error is clearly linear on the semi-logarithmic scale indicating the spectral accuracy. There is no saturation of the error. 

\begin{figure}
    \centering
     \includegraphics[width=0.8\textwidth]{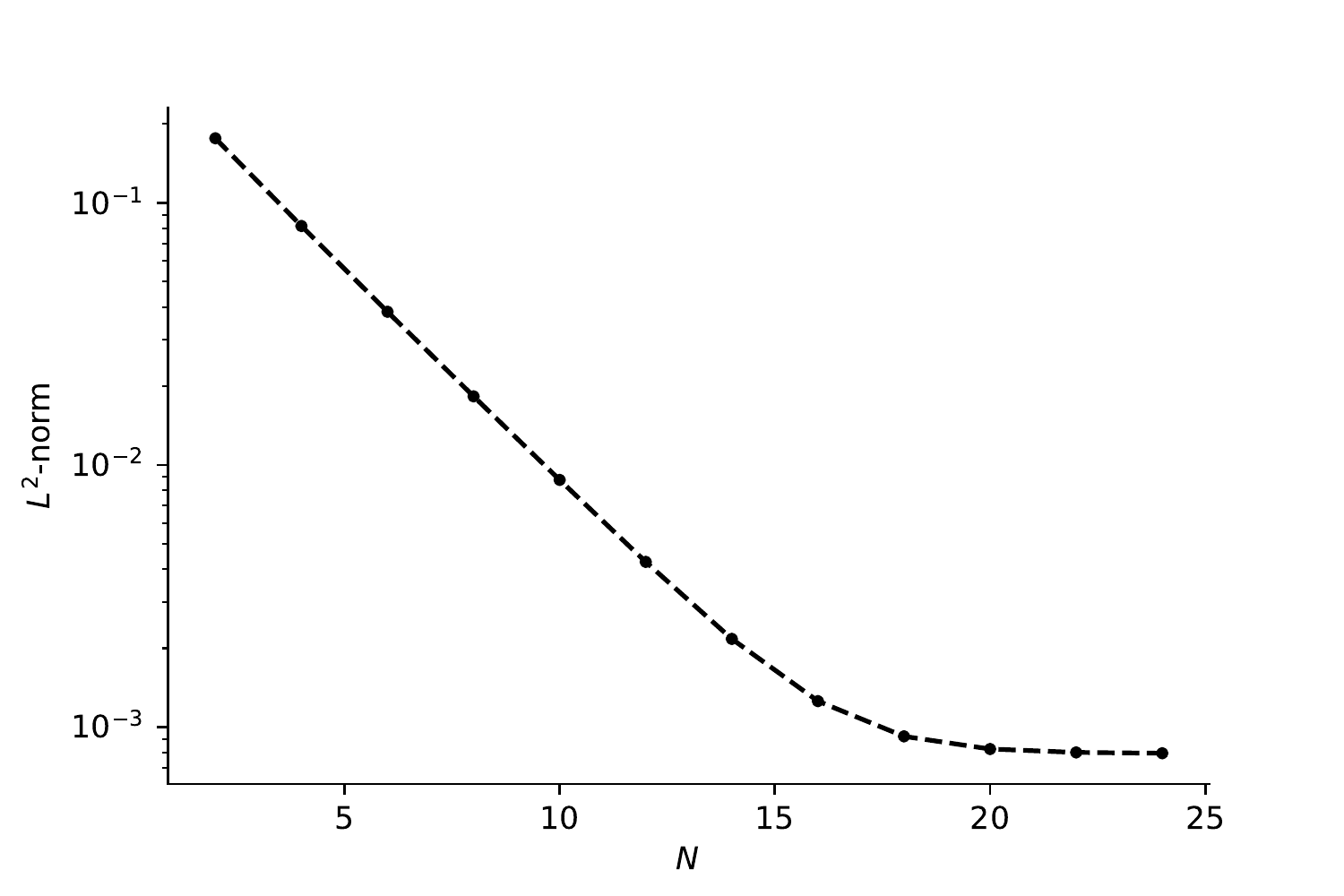}
    \caption{The error $\|U_N^n-u(t_n,x)\|$ with $\alpha=0.6,\ \beta=1$ and $ u(0,x)=e^{-x^2/2}$ calculated at point $t_n=1,\ n =2000$, with respect to number of Hermite functions in the expansion $N$. Note the semi-logarothmic scale }
    \label{fig:EK_Convergence_N}
\end{figure}
\begin{figure}
    \centering
     \includegraphics[width=0.8\textwidth]{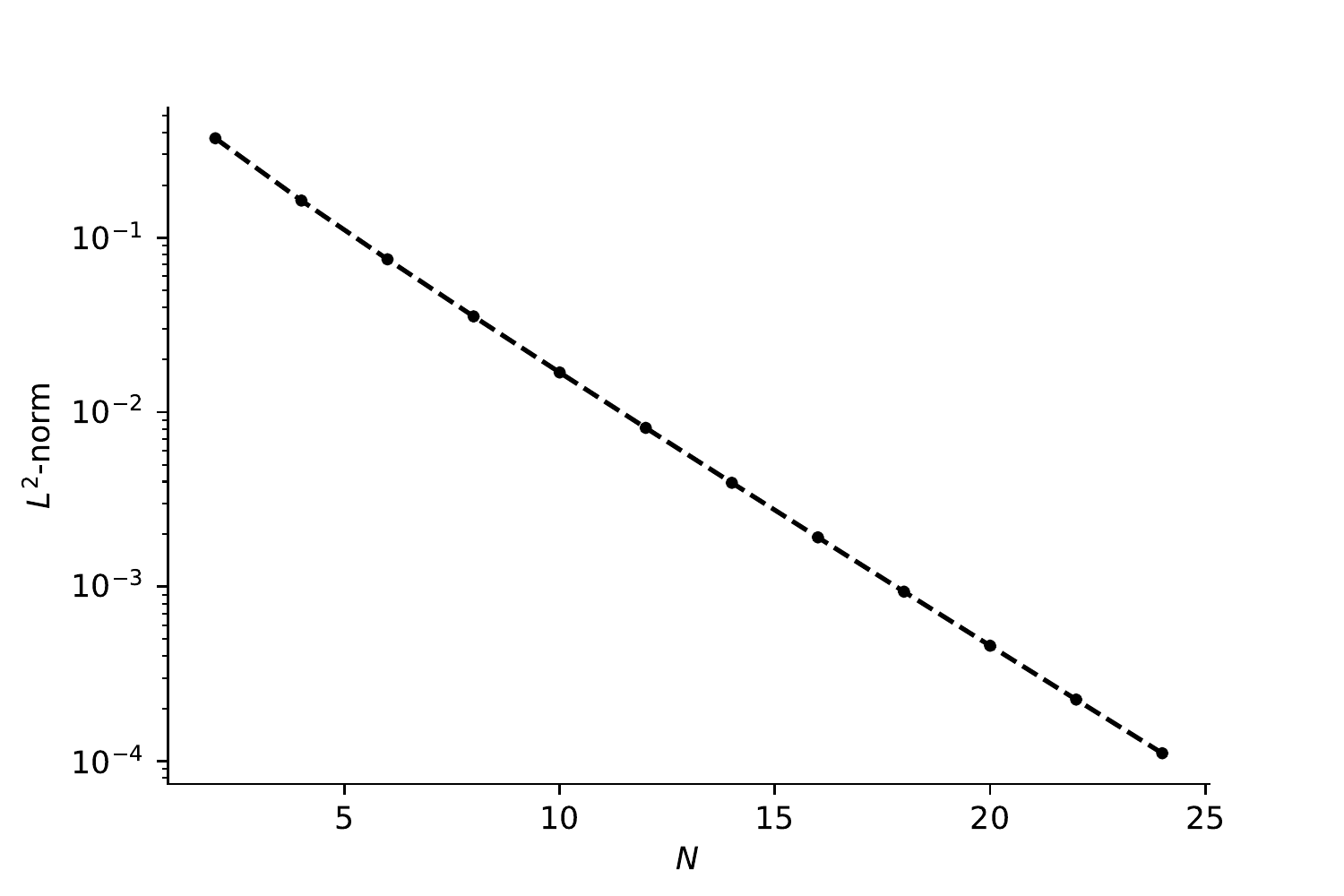}
    \caption{The error $\|U_N^n-U_{60}^n\|$ with $\alpha=0.6,\ \beta=1$ and $ u(0,x)=e^{-x^2/2}$ calculated at point $t_n=1,\ n =2000$, with respect to to number of Hermite functions in the expansion $N$. Note the semi-logarithmic scale. }
    \label{fig:EK_Convergence_N_num}
\end{figure}

\section{Conclusion}
\label{Conclusions}
The deterministic fractional diffusion equation describing the evolution of the marginal density function of particle dispersion of the generalized grey Brownian motion involves the Erd\'elyi--Kober fractional derivative. We proposed two discretization methods of this operator along with estimates of the error of approximation. Theoretical results were supported by numerical experiments. Furthermore, using the Galerkin--Hermite method, the numerical scheme for solving the Erd\'elyi--Kober fractional diffusion equation was proposed. For the semi-discrete problem with respect to time, we proved the stability and convergence. Due to the singular term in time present in the main equation, the error of the approximation is an order smaller than $1$. Resolving this issue, i.e., providing some higher order methods is one of the objectives of our future studies along with investigations concerning the existence, uniqueness, and regularity of the Erd\'elyi--Kober fractional diffusion equation. 

\section{Acknowledgments}
The research of M\'S was partially supported by NCN Sonata Bis Grant no. 2019/34/E/ST1/00360

\bibliography{biblio_erdelyi_kober}
\bibliographystyle{abbrvnat}
\end{document}